\documentclass[a4paper,10pt]{amsart}
\usepackage{graphicx}
\usepackage{amsmath,amssymb,xy,array}
\usepackage[T1]{fontenc}
\usepackage{amsfonts}
\usepackage{mathabx}
\usepackage{amsmath}
\usepackage{amssymb}
\usepackage{amsthm}
\usepackage{graphicx}
\usepackage{enumitem}
\usepackage{xcolor}
\usepackage{verbatim}
\usepackage[utf8]{inputenc}
\usepackage{hyperref}
\usepackage{mathrsfs}
\usepackage[toc,page]{appendix}
\usepackage{fancyhdr}
\usepackage{tikz}
\usepackage{tikz-cd}
\usepackage{tkz-euclide}
\usetkzobj{all}
\usepackage{longtable}
\usepackage{geometry}
\usepackage{textcomp}
\usetikzlibrary{trees}
\usetikzlibrary{arrows}
\usepackage{multirow}
\usepackage{youngtab}
\usepackage{mathdots}
\usepackage{makecell}

\definecolor{gray0}{gray}{0.40}
\definecolor{gray1}{gray}{0.60}
\definecolor{gray2}{gray}{0.70}

\setlist[itemize]{leftmargin=6mm}

\newcommand{\G}{\mathbb G}

\DeclareMathOperator{\codim}{codim}
\newcommand{\rddots}{\reflectbox{$\ddots$}}

\DeclareMathOperator{\Cl}{Cl}

\DeclareMathOperator{\NE}{NE}

\DeclareMathOperator{\rk}{rk}

\DeclareMathOperator{\lin}{lin}

\DeclareMathOperator{\mult}{mult}

\DeclareMathOperator{\Exc}{Exc}

\DeclareMathOperator{\Sing}{Sing}

\DeclareMathOperator{\Eff}{Eff}
\DeclareMathOperator{\Nef}{Nef}
\DeclareMathOperator{\Mov}{Mov}
\DeclareMathOperator{\Pic}{Pic}
\DeclareMathOperator{\rank}{rank}
\DeclareMathOperator{\MCD}{MCD}
\DeclareMathOperator{\SBLD}{SBLD}

\DeclareMathOperator{\mov}{mov}

\newtheorem{thm}{Theorem}[section]

\newtheorem{Lemma}[thm]{Lemma}
\newtheorem{Proposition}[thm]{Proposition}

\newtheorem{Corollary}[thm]{Corollary}

\theoremstyle{definition}

\newtheorem{Definition}[thm]{Definition}
\newtheorem{Remark}[thm]{Remark}

\newtheorem{Example}[thm]{Example}

\hypersetup{pdfpagemode=UseNone}
\hypersetup{pdfstartview=FitH}

\begin{document}

\title{On Mori chamber and stable base locus decompositions}

\author[Antonio Laface]{Antonio Laface}
\address{\sc Antonio Laface\\
Departamento de Matematica, Universidad de Concepci\'on\\
Casilla 160-C, Concepci\'on\\
Chile}
\email{alaface@udec.cl}

\author[Alex Massarenti]{Alex Massarenti}
\address{\sc Alex Massarenti\\ Dipartimento di Matematica e Informatica, Universit\`a di Ferrara, Via Machiavelli 30, 44121 Ferrara, Italy\newline
\indent Instituto de Matem\'atica e Estat\'istica, Universidade Federal Fluminense, Campus Gragoat\'a, Rua Alexandre Moura 8 - S\~ao Domingos\\
24210-200 Niter\'oi, Rio de Janeiro\\ Brazil}
\email{alex.massarenti@unife.it, alexmassarenti@id.uff.br}

\author[Rick Rischter]{Rick Rischter}
\address{\sc Rick Rischter\\
Universidade Federal de Itajub\'a (UNIFEI)\\ 
Av. BPS 1303, Bairro Pinheirinho\\ 
37500-903, Itajub\'a, Minas Gerais\\ 
Brazil}
\email{rischter@unifei.edu.br}

\date{\today}

\begin{abstract}
The effective cone of a Mori dream space admits two wall-and-chamber decompositions called Mori chamber and stable base locus decompositions. In general the former is a non trivial refinement of the latter. We investigate, from both the geometrical and the combinatorial viewpoints, the differences between these decompositions. Furthermore, we provide a criterion to establish whether the two decompositions coincide for a Mori dream space of Picard rank two, and we construct an explicit example of a Mori dream space of Picard rank two for which the decompositions are different, showing that our criterion is sharp. Finally, we classify the smooth toric $3$-folds of Picard rank three for which the two decompositions are different.
\end{abstract}

\maketitle 

\setcounter{tocdepth}{1}

\tableofcontents

\section{Introduction}
\textit{Mori dream spaces}, introduced by Y. Hu and S. Keel in \cite{HK00}, are varieties whose total coordinate ring, called the \textit{Cox ring}, is finitely generated. The birational geometry of a Mori dream space is encoded in its cone of effective divisors together with a chamber decomposition on it, called \textit{Mori chamber decomposition}. Two effective divisors lie in the interior of the same Mori chamber if there is an isomorphism between the target spaces of the corresponding dominant rational maps making the obvious triangular diagram commutative. 

The birational geometry of a Mori dream spaces can also be described via the Variation of Geometric Invariant Theory of its Cox ring. As proven in \cite{HK00}, \cite[Section 3.3.4]{ADHL15}, and \cite[Appendix A]{HKP06} in the case of complete toric varieties using the volume function, from this point of view GIT chambers correspond to Mori chambers.

The pseudo-effective cone of a projective variety with zero irregularity, so in particular of a Mori dream space, can be decomposed into chambers depending on the stable base locus of the corresponding linear series. Such decomposition, called \textit{stable base locus decomposition}, in general is coarser than the Mori chamber decomposition.

The Mori theory of important classes of moduli spaces such as moduli of curves \cite{Ha05}, \cite{HH09}, \cite{HH13}, Hilbert schemes of points on surfaces \cite{BC13}, \cite{ABCH13}, Kontsevich spaces of stable maps \cite{Ch08}, \cite{CC10}, \cite{CC11}, spaces of complete forms \cite{Ce15}, \cite{Ma18a}, \cite{Ma18b}, and moduli spaces of parabolic bundles \cite{Mu05}, \cite{AM16} have recently been studied in a series of papers. 

In this paper, given a Mori dream space $X$, we aim to understand how far is the stable base locus decomposition of $\Eff(X)$ from determining its Mori chamber decomposition. In Section \ref{examples} we produce examples of Mori dream spaces for which the two decompositions are different and we interpret them both from the geometric and the combinatorial viewpoints. 

While producing examples of either non compact varieties or of varieties with Picard rank greater than or equal to three turns out to be fairly feasible, it is quite tricky to exhibit a normal $\mathbb{Q}$-factorial projective Mori dream space of Picard rank two for which the Mori chamber decomposition is a non trivial refinement of the stable base locus decomposition. In Example \ref{ex1} we construct such a Mori dream space, and at the best of our knowledge this is the fist example of a projective variety displaying this particular behavior appearing in the literature. 

\begin{thm}\label{th1}
Let $Z$ be the toric variety with Cox ring
$K[T_1,\dots,T_{11}]$ whose 
grading matrix and irrelevant ideal are
the following
\[
 Q = \begin{bmatrix}
  1&1&2&2&2&2&1&0&0&0&0\\
  0&0&1&1&1&1&2&1&1&1&1
 \end{bmatrix}
 \qquad
  \mathcal J_{\rm irr}(Z) 
  = 
  \langle T_1,T_2\rangle
  \cap
  \langle T_3,\dots,T_{11}\rangle
\]
and let $F,G$ be two general polynomials of degree $(2,2)$ in the $T_i$. Then the ring 
$$\frac{K[T_1,\dots,T_{11}]}{(F,G)}$$
is the Cox ring of a projective normal $\mathbb{Q}$-factorial Mori dream space $X\subset Z$ of Picard rank two. Furthermore, the Mori chamber decomposition of $\Eff(X)$ consists of three chambers while its stable base locus decomposition consists of two chambers. 
\end{thm}

By Proposition \ref{prop_5chambers} if $X$ is a toric $3$-fold such that the two decompositions differ inside the movable cone then there are at least five Mori chambers in the movable cone, and Examples \ref{3toric_not_smooth}, \ref{4fold_toric} show that Proposition \ref{prop_5chambers} is sharp meaning that one can have that the two decompositions differ inside the movable cone with five chambers and with three chambers in dimension higher than three. In Section \ref{sec_toric3fold} we restrict to the smooth case, and we classify smooth toric $3$-folds of Picard rank three such that their Mori chamber and stable base locus decomposition do not coincide.

\begin{thm}\label{prop3foldpic3}
Let $X$ be a smooth toric $3$-fold of Picard rank three such that its Mori chamber and stable base locus decomposition do not coincide. Then the Mori chamber decomposition of $X$ is one of the seven types listed in the following table. 
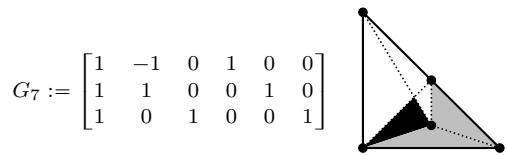
\begin{center}
\begin{footnotesize}
\begin{longtable}{cccc}
grading matrix& Effective cone \hspace*{1cm}&grading matrix& Effective cone \\ \hline
\makecell{\vspace{0.4cm} \\
$
 G_1 := 
 \begin{bmatrix}
   1&0&0&1&0&0\\
   \alpha&1&0&0&1&0\\
   \beta&0&1&0&0&1
 \end{bmatrix}
$\\ $\alpha,\beta >0$}& 
\makecell{\vspace{0.2cm} \\ 
\begin{tikzpicture}[scale=.3]
\tkzDefPoint(2,2){P1}
\tkzDefPoint(0,6){P2}
\tkzDefPoint(0,0){P3}
\tkzDefPoint(6,0){P4}
\tkzFillPolygon[color = black](P1,P2,P3)
\tkzFillPolygon[color = gray!50](P1,P2,P4)
\tkzFillPolygon[color = gray!50](P1,P3,P4)
\tkzDrawSegments[thick](P2,P3 P3,P4 P4,P2)
\tkzDrawSegments[densely dotted](P1,P2 P1,P3 P1,P4)
\tkzDrawPoints[fill=black,color=black,size=8](P1,P2,P3,P4)
\end{tikzpicture}\hspace*{1cm}}& 
\makecell{\vspace{0.4cm} \\ $  G_2 := 
 \begin{bmatrix}
   1&0&0&1&0&0\\
   \alpha&1&0&0&1&0\\
   \beta&\gamma&1&0&0&1
 \end{bmatrix} 
$ \\ $\alpha,\beta ,\gamma<0$}&  
\makecell{\vspace{0.2cm} \\ 
\begin{tikzpicture}[scale=.3]
\tkzDefPoint(2,1){P1}
\tkzDefPoint(0,6){P2}
\tkzDefPoint(0,0){P3}
\tkzDefPoint(6,0){P4}
\tkzDefPoint(0,3){P5}
\tkzInterLL(P5,P4)(P2,P1) \tkzGetPoint{Q1}
\tkzFillPolygon[color = black](P1,P3,P5)
\tkzFillPolygon[color = gray!50](P1,Q1,P4)
\tkzFillPolygon[color = gray!50](P1,P3,P4)
\tkzDrawSegments[thick](P2,P3 P3,P4 P4,P2)
\tkzDrawSegments[densely dotted](P1,P2 P1,P3 P1,P4 P5,P4 P5,P1)
\tkzDrawPoints[fill=black,color=black,size=8](P1,P2,P3,P4,P5)
\end{tikzpicture}} \vspace{0.2cm} \\ \hline
\makecell{\vspace{0.4cm} \\
$G_3 := 
 \begin{bmatrix}
   1&0&\gamma&1&0&0\\
   -1&1&0&0&1&0\\
   0&1&1&0&0&1
 \end{bmatrix} 
 $\\ $\gamma>0 $}& 
\makecell{\vspace{0.2cm} \\ 
\begin{tikzpicture}[scale=.3]
\tkzDefPoint(3,1){P1}
\tkzDefPoint(0,6){P2}
\tkzDefPoint(0,0){P3}
\tkzDefPoint(6,0){P4}
\tkzDefPoint(0,2){P5}
\tkzDefPoint(1,3){P6}
\tkzInterLL(P1,P2)(P4,P6) \tkzGetPoint{Q1}
\tkzInterLL(P6,P3)(P4,P5) \tkzGetPoint{Q2}
\tkzFillPolygon[color = black](P1,P6,Q2)
\tkzFillPolygon[color = gray!50](P5,P2,P6)
\tkzFillPolygon[color = gray!50](P2,P6,Q1)
\tkzDrawSegments[thick](P2,P3 P3,P4 P4,P2)
\tkzDrawSegments[densely dotted](P1,P2 P1,P3 P1,P4 P5,P4 P5,P1 P3,P6 P1,P6 P2,P6 P4,P6 P5,P6)
\tkzDrawPoints[fill=black,color=black,size=8](P1,P2,P3,P4,P5,P6)
\end{tikzpicture}\hspace*{1cm}
}& 
\makecell{\vspace{0.4cm} \\ $  G_4 := 
 \begin{bmatrix}
   1&\beta&\gamma&1&0&0\\
   -1&1&0&0&1&0\\
   0&1&1&0&0&1
 \end{bmatrix}
$ \\ $ \beta>0>\gamma$}&  
\makecell{\vspace{0.2cm} \\ 
\begin{tikzpicture}[scale=.3]
\tkzDefPoint(3.5,0){P1}
\tkzDefPoint(0,6){P2}
\tkzDefPoint(0,0){P3}
\tkzDefPoint(6,0){P4}
\tkzDefPoint(0,3){P5}
\tkzDefPoint(1.4,2.3){P6}
\tkzInterLL(P1,P2)(P4,P6) \tkzGetPoint{Q1}
\tkzInterLL(P5,P1)(P3,P6) \tkzGetPoint{Q2}
\tkzFillPolygon[color = black](Q2,P6,P1)
\tkzFillPolygon[color = gray!50](P5,P2,P6)
\tkzFillPolygon[color = gray!50](P2,P6,Q1)
\tkzDrawSegments[thick](P2,P3 P3,P4 P4,P2)
\tkzDrawSegments[densely dotted](P1,P2 P1,P3 P1,P4 P5,P4 P5,P1 P3,P6 P1,P6 P2,P6 P4,P6 P5,P6)
\tkzDrawPoints[fill=black,color=black,size=8](P1,P2,P3,P4,P5,P6)
\end{tikzpicture}} \vspace{0.2cm}\\ \hline
\makecell{\vspace{0.4cm} \\
$
  G_5 := 
 \begin{bmatrix}
   1&\beta&\gamma&1&0&0\\
   -1&1&0&0&1&0\\
   0&1&1&0&0&1
 \end{bmatrix}
$\\ $\beta<0<\gamma$}& 
\makecell{\vspace{0.2cm} \\ 
\begin{tikzpicture}[scale=.3]
\tkzDefPoint(1.5,2){P1}
\tkzDefPoint(0,6){P2}
\tkzDefPoint(0,0){P3}
\tkzDefPoint(6,0){P4}
\tkzDefPoint(0,2){P5}
\tkzDefPoint(3,2){P6}
\tkzInterLL(P1,P3)(P4,P5) \tkzGetPoint{Q1}
\tkzInterLL(P4,P1)(P3,P6) \tkzGetPoint{Q2}
\tkzInterLL(P4,P5)(P3,P6) \tkzGetPoint{Q3}
\tkzFillPolygon[color = black](P1,Q1,Q3,Q2)
\tkzFillPolygon[color = gray!50](P5,P2,P1)
\tkzFillPolygon[color = gray!50](P2,P6,P1)
\tkzDrawSegments[thick](P2,P3 P3,P4 P4,P2)
\tkzDrawSegments[densely dotted](P1,P2 P1,P3 P1,P4 P5,P4 P5,P1 P3,P6 P1,P6 P2,P6 P4,P6 P5,P6)
\tkzDrawPoints[fill=black,color=black,size=8](P1,P2,P3,P4,P5,P6)
\end{tikzpicture}\hspace*{1cm}}& 
\makecell{\vspace{0.4cm} \\ $  G_6 := 
 \begin{bmatrix}
   1&\beta&\gamma&1&0&0\\
   -1&1&0&0&1&0\\
   0&1&1&0&0&1
 \end{bmatrix}
$ \\ $\beta<\gamma-1<-1$}&  
\makecell{\vspace{0.2cm} \\ 
\begin{tikzpicture}[scale=.3]
\tkzDefPoint(1.5,2){P1}
\tkzDefPoint(0,6){P2}
\tkzDefPoint(0,0){P3}
\tkzDefPoint(6,0){P4}
\tkzDefPoint(0,2){P5}
\tkzDefPoint(4,2){P6}
\tkzInterLL(P1,P3)(P4,P5) \tkzGetPoint{Q1}
\tkzInterLL(P4,P1)(P3,P6) \tkzGetPoint{Q2}
\tkzInterLL(P4,P5)(P3,P6) \tkzGetPoint{Q3}
\tkzFillPolygon[color = black](P1,Q1,Q3,Q2)
\tkzFillPolygon[color = gray!50](P5,P2,P1)
\tkzFillPolygon[color = gray!50](P2,P6,P1)
\tkzDrawSegments[thick](P2,P3 P3,P4 P4,P2)
\tkzDrawSegments[densely dotted](P1,P2 P1,P3 P1,P4 P5,P4 P5,P1 P3,P6 P1,P6 P2,P6 P4,P6 P5,P6)
\tkzDrawPoints[fill=black,color=black,size=8](P1,P2,P3,P4,P5,P6)
\end{tikzpicture}} \vspace{0.2cm}\\
\hline
\makecell{\vspace{0.4cm} \\
$
G_7 := 
 \begin{bmatrix}
   1&-1&0&1&0&0\\
   1&1&0&0&1&0\\
   1&0&1&0&0&1
 \end{bmatrix} 
$\\ }&
\makecell{\vspace{0.2cm} \\ 
\begin{tikzpicture}[scale=.3]
\tkzDefPoint(3,3){P1}
\tkzDefPoint(0,6){P2}
\tkzDefPoint(0,0){P3}
\tkzDefPoint(6,0){P4}
\tkzDefPoint(3,1){P5}
\tkzInterLL(P1,P3)(P2,P5) \tkzGetPoint{Q1}
\tkzFillPolygon[color = black](P3,P5,Q1)
\tkzFillPolygon[color = gray!50](P5,P4,P1)
\tkzFillPolygon[color = gray!50](P3,P4,P5)
\tkzDrawSegments[thick](P2,P3 P3,P4 P4,P2)
\tkzDrawSegments[densely dotted](P1,P3 P1,P5 P2,P5 P4,P5)
\tkzDrawPoints[fill=black,color=black,size=8](P1,P2,P3,P4,P5)
\end{tikzpicture}\hspace*{1cm}}& &
\label{table_toric3fold}
\end{longtable}
\end{footnotesize}
\end{center}
Where, in the pictures, the black region is the semi-ample cone and the two gray regions are the GIT chambers sharing the same stable base locus. In particular, for any smooth toric $3$-fold the Mori chamber and the stable base locus decomposition coincide inside the movable cone.
\end{thm} 

In Section \ref{mainSec} we focus on Mori dream spaces of Picard rank two. Recall that a Mori dream space can be recovered as a GIT quotient, with respect to a suitable polarization, of the spectrum of its Cox ring by a torus. In the fist part of Section \ref{mainSec}, assuming that the Picard rank is two, we reach a simple description of the non semi-stable loci with respect to all the possible polarizations, and of the stable base loci of effective divisors in terms of the generators of the Cox ring. Thanks to these characterizations in Theorem \ref{main} we get technical criteria on the non semi-stable loci aimed to establish whether the Mori chamber and the stable base locus decomposition of a given Mori dream space of Picard rank two coincide. As observed in Corollary \ref{irr} the irreducibility of the non semi-stable loci is a sufficient condition for the two decompositions to coincide. 

In Theorem \ref{main2} we prove that under suitable inequalities, that need just the knowledge of the generators of the Cox ring in order to be checked, the two decompositions coincide. Furthermore, in Proposition \ref{crit1} we get another criterion for the equality of the decompositions. The usefulness of these results lies in the fact that in general, even in Picard rank two, the stable base locus decomposition is considerably easier to compute than the Mori chamber decomposition.

Note that if $X$ is a projective Mori dream space of Picard rank two we can fix a total order on the classes in the effective cone: $w\leq w'$ if $w$ is on the left of $w'$. Given two convex cones $\lambda,\lambda'$ contained in the effective cone we will write $\lambda\leq \lambda'$ if $w\leq w'$ for any $w\in\lambda$ and $w'\in\lambda$. Denote by $\{f_1,\dots,f_r\}$ a minimal set of homogeneous generators for the Cox ring $\mathcal{R}(X)$ of $X$, and let $w_i = \deg(f_i)$ for any $i$.

The criteria in the Proposition \ref{crit1}, Theorem \ref{main2} and Corollary \ref{hyp} can be summarized in the following statement.

\begin{thm}\label{th2}
Let $X$ be a $\mathbb{Q}$-factorial Mori dream space with Picard rank two, $\{f_1,\dots,f_r\}$ a minimal set of homogeneous generators for the Cox ring $\mathcal{R}(X)$, $w_i := \deg(f_i)$, and $\lambda_A$ be the ample chamber of $X$. Denote by $c$ the codimension of $X$ into its canonical toric embedding~\cite[Section 3.2.5]{ADHL15}. Define 
$$h^+:=\#\{f_i\, :\, w_i\geq\lambda_A\}\quad and \quad h^-:=\#\{f_i\, :\, w_i\leq\lambda_A\}$$ 
If one of the following two conditions is satisfied 
\begin{itemize}
\item[(i)] all the generators of $\mathcal{R}(X)$ appear in the walls of the stable base locus decomposition of $\Eff(X)$,
\item[(ii)] $h^->c$ and $h^+>c$, 
\end{itemize} 
then the Mori chamber decomposition and the stable base locus decomposition of $\Eff(X)$ coincide.

In particular, if $Z$ is a projective normal $\mathbb{Q}$-factorial toric variety with $\rk(\Cl(Z))=2$, and $X\subseteq Z$ is a projective normal $\mathbb{Q}$-factorial Mori dream hypersurface such that $\imath^*\colon {\Cl}(Z)\to{\Cl}(X)$ is an isomorphism, then the Mori chamber and the stable base locus decompositions of both $\Eff(Z)$ and $\Eff(X)$ coincide.
\end{thm}
As observed in Remark \ref{sharp}, Theorem \ref{th1} shows that the bounds in Theorem \ref{th2} item $(ii)$ can not be improved. Indeed in Theorem \ref{th1} we have $h^+=c=2$.

In Subsection \ref{comp1} we show how Theorem \ref{th2} item $(ii)$, together with the classification of Picard rank two varieties, with a torus action of complexity one in \cite[Theorem 1.1]{FHN16} immediately implies that the Mori chamber decomposition is equal to the stable base locus decomposition for this class of varieties. Note that, as shown by the examples in Subsection \ref{comp1}, Lemma \ref{SBLc} along with the proof of Theorem \ref{main} provide a concrete method to compute Mori chamber decompositions.     

We would like so stress that these results can also be useful in order to compute the Sarkisov factorization of a birational map $X\dasharrow Y$ between two $\mathbb{Q}$-factorial Fano varieties of Picard rank one. Indeed, if there exists a Mori dream space $Z$ of Picard rank two admitting a dominant morphism $Z\rightarrow X$ then such a factorization is determined by a so called $2$-ray game on $Z$, and such a $2$-ray game is in turn determined by the Mori chamber decomposition of $\Eff(Z)$. We refer to \cite{Co95}, \cite{HM13}, \cite{AZ16}, \cite{Ah17} for details on this topic and explicit examples.

Finally, in Section \ref{grassbu} we apply Theorem \ref{th2} item $(i)$ to show that the Mori chamber decomposition of the blow-up $\G(r,n)_1$ of the Grassmannian $\G(r,n)$, parametrizing $r$-planes in $\mathbb{P}^n$, at a point coincides with its stable base locus decomposition, and can be described in terms of linear systems of hyperplanes containing the osculating spaces of $\G(r,n)$ at the blown-up point. This provides a positive answer to \cite[Question 6.9]{MR18}. 

All through the paper we will work over an algebraically closed field $K$ of characteristic zero, and given a $\mathbb{Q}$-factorial Mori dream space $X$ we will denote by $\MCD(X)$ and $\SBLD(X)$ respectively the Mori chamber decomposition and the stable base locus decomposition of its effective cone. 

\subsection*{Acknowledgments}
We thank the referee for giving us 
suggestions that lead us to Section
~\ref{sec_toric3fold}, Proposition
~\ref{re:sbl}, Proposition~\ref{prop_5chambers}
and Examples~\ref{3toric_not_smooth}
and ~\ref{4fold_toric}.

The first named author was partially supported 
by Proyecto FONDECYT Regular N. 1150732,
Proyecto FONDECYT Regular N. 1190777
and by project Anillo ACT 1415 PIA Conicyt.
The second named author is a member of the Gruppo Nazionale per le Strutture Algebriche, Geometriche e le loro Applicazioni of the Istituto Nazionale di Alta Matematica "F. Severi" (GNSAGA-INDAM).

\section{Mori chamber and stable base locus decompositions}
Let $X$ be a normal projective variety over an algebraically closed field of characteristic zero. We denote by $N^1(X)$ the real vector space of $\mathbb{R}$-Cartier divisors modulo numerical equivalence. 
The \emph{nef cone} of $X$ is the closed convex cone $\Nef(X)\subset N^1(X)$ generated by classes of 
nef divisors. 
The \emph{movable cone} of $X$ is the convex cone $\Mov(X)\subset N^1(X)$ generated by classes of 
\emph{movable divisors}. These are Cartier divisors whose stable base locus has codimension at least two in $X$.
The \emph{effective cone} of $X$ is the convex cone $\Eff(X)\subset N^1(X)$ generated by classes of 
\emph{effective divisors}. We have inclusions $\Nef(X)\ \subset \ \overline{\Mov(X)}\ \subset \ \overline{\Eff(X)}$. 

We will denote by $N_1(X)$ be the real vector space of numerical equivalence classes of $1$-cycles on $X$. The closure of the cone in $N_1(X)$ generated by the classes of irreducible curves in $X$ is called is called the \textit{Mori cone} of $X$, we will denote it by $\NE(X)$. 

A class $[C]\in N_1(X)$ is called \textit{moving} if the curves in $X$ of class $[C]$ cover a dense open subset of $X$. The closure of the cone in $N_1(X)$ generated by classes of moving curves in $X$ is called the \textit{moving cone} of $X$ and we will denote it by $\mov(X)$.  
We refer to \cite[Chapter 1]{De01} for a comprehensive treatment of these topics. 

We say that a birational map  $f: X \dasharrow X'$ to a normal projective variety $X'$  is a \emph{birational contraction} if its
inverse does not contract any divisor. 
We say that it is a \emph{small $\mathbb{Q}$-factorial modification} 
if $X'$ is $\mathbb{Q}$-factorial  and $f$ is an isomorphism in codimension one.
If  $f: X \dasharrow X'$ is a small $\mathbb{Q}$-factorial modification, then 
the natural pull-back map $f^*:N^1(X')\to N^1(X)$ sends $\Mov(X')$ and $\Eff(X')$
isomorphically onto $\Mov(X)$ and $\Eff(X)$, respectively.
In particular, we have $f^*(\Nef(X'))\subset \overline{\Mov(X)}$.

\begin{Definition}\label{def:MDS} 
A normal projective $\mathbb{Q}$-factorial variety $X$ is called a \emph{Mori dream space}
if the following conditions hold:
\begin{enumerate}
\item[-] $\Pic{(X)}$ is finitely generated, or equivalently $h^1(X,\mathcal{O}_X)=0$,
\item[-] $\Nef{(X)}$ is generated by the classes of finitely many semi-ample divisors,
\item[-] there is a finite collection of small $\mathbb{Q}$-factorial modifications
 $f_i: X \dasharrow X_i$, such that each $X_i$ satisfies the second condition above, and $
 \Mov{(X)} \ = \ \bigcup_i \  f_i^*(\Nef{(X_i)})$.
\end{enumerate}
\end{Definition}
By \cite[Corollary 1.3.2]{BCHM10} smooth Fano varieties are Mori dream spaces. In fact, there is a larger class of varieties called log Fano varieties which are Mori dream spaces as well. By the work of M. Brion \cite{Br93} we have that $\mathbb{Q}$-factorial spherical varieties are Mori dream spaces. An alternative proof of this result can be found in \cite[Section 4]{Pe14}. 

The collection of all faces of all cones $f_i^*(\Nef{(X_i)})$ in Definition \ref{def:MDS} forms a fan which is supported on $\Mov(X)$.
If two maximal cones of this fan, say $f_i^*(\Nef{(X_i)})$ and $f_j^*(\Nef{(X_j)})$, meet along a facet,
then there exist a normal projective variety $Y$, a small modification $\varphi:X_i\dasharrow X_j$, and $h_i:X_i\rightarrow Y$ and $h_j:X_j\rightarrow Y$ small birational morphisms of relative Picard number one such that $h_j\circ\varphi = h_i$. The fan structure on $\Mov(X)$ can be extended to a fan supported on $\Eff(X)$ as follows. 

\begin{Definition}\label{MCD}
Let $X$ be a Mori dream space.
We describe a fan structure on the effective cone $\Eff(X)$, called the \emph{Mori chamber decomposition}.
We refer to \cite[Proposition 1.11]{HK00} and \cite[Section 2.2]{Ok16} for details.
There are finitely many birational contractions from $X$ to Mori dream spaces, denoted by $g_i:X\dasharrow Y_i$.
The set $\Exc(g_i)$ of exceptional prime divisors of $g_i$ has cardinality $\rho(X/Y_i)=\rho(X)-\rho(Y_i)$.
The maximal cones $\mathcal{C}$ of the Mori chamber decomposition of $\Eff(X)$ are of the form: $\mathcal{C}_i \ = \left\langle g_i^*\big(\Nef(Y_i)\big) , \Exc(g_i) \right\rangle$. We call $\mathcal{C}_i$ or its interior $\mathcal{C}_i^{^\circ}$ a \emph{maximal chamber} of $\Eff(X)$.
\end{Definition}

\begin{Definition}
Let $X$ be a normal projective variety with finitely generated divisor class group $\Cl(X) := {\rm WDiv}(X)/{\rm PDiv}(X)$, in particular $h^1(X,\mathcal O_X)=0$. The \textit{Cox sheaf} and {\em Cox ring}
of $X$ are defined as
\[
 \mathcal R := \bigoplus_{[D]\in \Cl(X)}\mathcal{O}_X(D)
 \qquad
 \qquad
 \mathcal R(X) := \Gamma(X,\mathcal R)
\]
\end{Definition}

Recall that $\mathcal R$ is a sheaf of
${\rm Cl}(X)$-graded $\mathcal O_X$-algebras, 
whose multiplication maps are discussed in
~\cite[Section 1.4]{ADHL15}. In case the divisor
class group is torsion-free one can just 
take the direct sum over a subgroup of
${\rm WDiv}(X)$, isomorphic to ${\rm Cl}(X)$ 
via the quotient map, getting immediately a sheaf of $\mathcal O_X$-algebras.
Denote by $\widehat X$ the 
relative spectrum of $\mathcal R$ 
and by $\overline X$ the spectrum
of $\mathcal{R}(X)$. The ${\rm Cl}(X)$-grading
induces an action of the quasi-torus 
$H_X := {\rm Spec}\,\mathbb C[{\rm Cl}(X)]$
on both spaces. The inclusion 
$\mathcal O_X\to\mathcal R$ induces 
a good quotient $p_X\colon \widehat X\to X$
with respect to this action.
Summarizing we have the following	
diagram
  \[
  \begin{tikzpicture}[xscale=0.65,yscale=-1.2]
    \node (A0_0) at (0, 0) {$\widehat{X}$};
    \node (A0_1) at (1, 0) {$\subseteq\overline{X}$};
    \node (A1_0) at (0, 1) {$X$};
    \path (A0_0) edge [->,swap]node [auto] {$\scriptstyle{p_X}$} (A1_0);
  \end{tikzpicture}
  \]
to which we will refer as the {\em Cox 
construction} of $X$. In case ${\mathcal R}(X)$
is a finitely generated algebra the complement
of $\widehat X$ in the affine variety 
$\overline X$ has codimension $\geq 2$.
This subvariety is the {\em irrelevant locus}
and its defining ideal is the {\em irrelevant
ideal} $\mathcal J_{\rm irr}(X) \subseteq 
\mathcal R(X)$.

\begin{Remark}\label{dimCox}
By \cite[Proposition 2.9]{HK00} a normal and $\mathbb{Q}$-factorial projective
variety $X$ over an algebraically closed field $K$, with finitely generated Picard group is a Mori dream space if and only if $\mathcal{R}(X)$ is a finitely generated $K$-algebra. Furthermore, the following equality holds
$$\dim\mathcal{R}(X) = \dim(X)+\rank\Cl(X)$$
see for instance \cite[Theorem 3.2.1.4]{ADHL15}.
\end{Remark}

Let $X$ be a normal $\mathbb{Q}$-factorial projective variety, and let $D$ be an effective $\mathbb{Q}$-divisor on $X$. The stable base locus $\textbf{B}(D)$ of $D$ is the set-theoretic intersection of the base loci of the complete linear systems $|sD|$ for all positive integers $s$ such that $sD$ is integral
$$\textbf{B}(D) = \bigcap_{s > 0}B(sD).$$
Since stable base loci do not behave well with respect to numerical equivalence, we will assume that $h^{1}(X,\mathcal{O}_X)=0$ so that linear and numerical equivalence of $\mathbb{Q}$-divisors coincide. 

Then numerically equivalent $\mathbb{Q}$-divisors on $X$ have the same stable base locus, and the pseudo-effective cone $\overline{\Eff}(X)$ of $X$ can be decomposed into chambers depending on the stable
base locus of the corresponding linear series called \textit{stable base locus decomposition}, see \cite[Section 4.1.3]{CdFG17} for further details. 

If $X$ is a Mori dream space, satisfying then the condition $h^1(X,\mathcal{O}_X)=0$, determining the stable base locus decomposition of $\Eff(X)$ is a first step in order to compute its Mori chamber decomposition.

\begin{Remark}\label{SBLMC}
Recall that two divisors $D_1,D_2$ are said to be \textit{Mori equivalent} if $\textbf{B}(D_1) = \textbf{B}(D_2)$ and the following diagram of rational maps is commutative
   \[
  \begin{tikzpicture}[xscale=1.5,yscale=-1.2]
    \node (A0_1) at (1, 0) {$X$};
    \node (A1_0) at (0, 1) {$X(D_1)$};
    \node (A1_2) at (2, 1) {$X(D_2)$};
    \path (A1_0) edge [->]node [auto] {$\scriptstyle{}$} node [rotate=180,sloped] {$\scriptstyle{\widetilde{\ \ \ }}$} (A1_2);
    \path (A0_1) edge [->,dashed]node [auto] {$\scriptstyle{\phi_{D_2}}$} (A1_2);
    \path (A0_1) edge [->,swap, dashed]node [auto] {$\scriptstyle{\phi_{D_1}}$} (A1_0);
  \end{tikzpicture}
  \]
where the horizontal arrow is an isomorphism. Therefore, the Mori chamber decomposition is a refinement of the stable base locus decomposition.
\end{Remark}

Let $X$ be a Mori dream space with Cox ring
$\mathcal{R}(X)$ and grading matrix $Q$. The matrix $Q$ 
defines a surjection
\[
 Q\colon E\to{\rm Cl}(X)
\]
from a free, finitely generated, 
abelian group $E$ to the divisor 
class group of $X$.
Denote by $\gamma$ the positive quadrant of 
$E_{\mathbb Q} := E\otimes_{\mathbb Z}\mathbb Q$.
Let $e_1,\dots,e_r$ be the canonical 
basis of $E_{\mathbb Q}$.
Given a face $\gamma_0\preceq\gamma$ we say that $i\in\{1,\dots,r\}$ is a cone index 
of $\gamma_0$ if $e_i\in \gamma_0$.
The face $\gamma_0$ is an {\em $\mathfrak F$-face}
if there exists a point of $\overline X 
= {\rm Spec}(\mathcal{R}(X))$ whose $i$-th coordinate
is non-zero exactly when $i$ is a cone index 
of $\gamma_0$ \cite[Construction 3.3.1.1]{ADHL15}.
The set of these points is denoted by 
$\overline X(\gamma_0)$.
\begin{Example}
If $\mathcal{R}(X) = \frac{K[T_1,\dots,T_5]}{
\langle T_1T_2+T_3^2+T_4T_5\rangle}$
then $\gamma_0 = {\rm cone}(e_1,e_4)$ is an 
$\mathfrak F$-face and 
$$\overline X(\gamma_0) = \{(x_1,0,0,x_4,0)\in\overline X\, :\, x_1x_4\neq 0\}$$
On the other hand, ${\rm cone}(e_1,e_3,e_4)$ is not an $\mathfrak F$-face.
\end{Example}
Given the Cox construction of $X$ we denote by $X(\gamma_0)\subseteq X$ the 
image of $\overline X(\gamma_0)$, and given an $\mathfrak F$-face $\gamma_0$
its image $Q(\gamma_0)\subseteq {\rm Cl}(X)_{\mathbb Q}$
is an {\em orbit cone} of $X$. The set of all orbit
cones of $X$ is denoted by $\Omega$.
Accordingly to \cite[Definition 3.1.2.6]{ADHL15}
a class $w\in {\rm Cl}(X)$ defines the {\em GIT
chamber}
\stepcounter{thm}
\begin{equation}\label{gitch}
\lambda(w):=\bigcap_{\{\omega\in\Omega\, :\, w\in\omega\}}\omega
\end{equation}
If $w$ is an ample class of $X$ the corresponding
GIT chamber is the semi-ample cone of $X$.
The variety $X$ can be reconstructed from
the pair $(\mathcal{R}(X),\Phi)$ formed by
the Cox ring together with a {\em bunch of cones},
consisting of certain subsets of the orbit cones
~\cite[Definition 3.1.3.2]{ADHL15}. 
According to \cite[Example 3.1.3.6]{ADHL15} every
GIT chamber $\lambda$ defines a 
{\em bunch of orbit cones}
\[
 \Phi(\lambda)
 :=
 \{\omega\in\Omega\, :\, \omega^\circ\supseteq\lambda^\circ\}
\]
Given a class $w\in{\rm Cl}(X)$ we denote
by $\lambda^{\rm sbl}(w)$ the subset of 
${\rm Cl}(X)_{\mathbb Q}$ consisting of all
classes which have the same stable base
locus of $w$.

\begin{Proposition}\label{sbl}
Let $X$ be a normal variety with finitely 
generated Cox ring, bunch of orbit 
cones $\Phi$, and let $w\in {\rm Cl}(X)$ 
be a class of $X$. Then
\[
 \lambda^{\rm sbl}(w)
 =
 \bigcap_{\{\omega\in\Phi\, :\, w\in \omega\}}\omega
 \cap
 \bigcap_{\{\omega\in\Phi\, :\, w\notin \omega\}}
 \omega^c
\]
\end{Proposition}
\begin{proof}
Recall that, according to \cite[Construction 3.2.1.3]{ADHL15},
the set of relevant faces ${\rm rlv}(\Phi)$ is the set of 
faces of $\gamma$ which are mapped by $Q$ to 
elements of $\Phi$. 
Each relevant face $\gamma_0\preceq\gamma$ defines
a subset $X(\gamma_0)\subseteq X$ consisting of 
all the points of $X$ whose $i$-th Cox coordinate
is non-zero exactly when $i$ is a cone index 
of $\gamma_0$ \cite[Construction 3.3.1.1]{ADHL15}.
By \cite[Proposition 3.3.2.8]{ADHL15}
the stable base locus of a class $w$ is the
union
\stepcounter{thm}
\begin{equation}
\label{eq-sbl}
 {\bf B}(w) := 
 \bigcup_{\{\gamma_0\in{\rm rlv}(\Phi)\, :\, w\notin Q(\gamma_0)\}}X(\gamma_0)
\end{equation}
Applying $Q$ to the elements of the set
$\{\gamma_0\in{\rm rlv}(\Phi)\, :\, w\notin Q(\gamma_0)\}$
one gets the set $\{\omega\in\Phi\, :\, w\notin \omega\}$
and the former set is completely determined by
the latter.
We claim that two classes $w,w'$ define
the same stable base locus if and only if
the following holds
\stepcounter{thm}
\begin{equation}
\label{equality}
 \{\omega\in\Phi\, :\, w\in \omega\}
 =
 \{\omega\in\Phi\, :\, w'\in \omega\}
\end{equation}
Clearly, if $w,w'$ define
the same stable base locus then (\ref{equality}) holds. Now, assume that there is an $\omega\in\Phi$
such that $w'\notin\omega$ and $w\in\omega$.
Let $\gamma_0\preceq\gamma$ be such that
$Q(\gamma_0) = \omega$. It suffices to show 
that $X(\gamma_0)$ is not contained in 
the stable base locus of $w$.
Indeed if $X(\gamma_1)$ is any strata which
contains $X(\gamma_0)$ then 
$\gamma_0\preceq\gamma_1$, so that
$w\in\omega = Q(\gamma_0) \subseteq 
Q(\gamma_1)$. Thus  $X(\gamma_1)$ does
not appear in~\eqref{eq-sbl}, and the claim is proved. Finally, the statement follows by observing that if
\[
 w'\in
 \bigcap_{\{\omega\in\Phi\, :\, w\in \omega\}}\omega
 \cap
 \bigcap_{\{\omega\in\Phi\, :\, w\notin \omega\}}
 \omega^c
\]
then the cones which contain, respectively
do not contain $w'$ are the same of those
which contain, respectively do not contain 
$w$.
\end{proof}

\begin{Corollary}
Let $X$ be a normal variety with finitely 
generated Cox ring, bunch of orbit 
cones $\Phi$ and let $w\in {\rm Cl}(X)$ 
be a class of $X$. Then the following inclusion
holds
\[
 \lambda(w) \subseteq \lambda^{{\rm sbl}}(w)
\]
\end{Corollary}
\begin{proof}
It is a direct consequence of Proposition \ref{sbl} and of the following equalities
\[
 \lambda(w) 
  =  \bigcap_{\{\omega\in\Omega\, :\, w\in\omega\}}\omega
  =  \bigcap_{\{\omega\in\Omega\, :\, w\in\omega\}}\omega
 \cap \bigcap_{\{\omega\in\Omega\, :\, w\notin\omega\}}\omega^c
\]
where the first is by definition while the second
is due to the fact that any two classes in the
relative interior $\lambda(w)$ determine the
same chamber.
\end{proof}

\begin{Proposition}
\label{re:sbl}
Let $X$ be a normal variety with finitely 
generated Cox ring, bunch of orbit 
cones $\Phi$, and $w_1,w_2\in {\rm Cl}(X)$. Then the following are 
equivalent:
\begin{enumerate}
\item
$\lambda^{\rm sbl}(w_1) = \lambda^{\rm sbl}(w_2)$;
\item
$\{\omega\in\Phi\, :\, w_1\in \omega\}
 =
 \{\omega\in\Phi\, :\, w_2\in \omega\}$;
\item
$\bigcap_{\{\omega\in\Phi\, :\, w_1\in \omega\}}\omega
 =
\bigcap_{\{\omega\in\Phi\, :\, w_2\in \omega\}}\omega$.
\end{enumerate}
Moreover, if $\Phi = \Phi(\lambda)$, with 
$\lambda$ distinct from $\lambda_1$ and
$\lambda_2$, then each of the above 
condition is implied by 
${\rm cone}(\lambda\cup\lambda_1) 
\cap\mathring{\lambda}_2\neq 0$
and ${\rm cone}(\lambda\cup\lambda_2) 
\cap\mathring{\lambda}_1\neq 0$.
\end{Proposition}
\begin{proof}
The equivalence of $(1)$ with $(2)$ is
given by the proof of Proposition
~\ref{sbl}, and the implication $(2)\Rightarrow (3)$
is clear. To prove $(3)\Rightarrow (2)$
it suffices to observe
that if there is an $\omega\in\Omega$
such that $w_1\in\omega$ but $w_2\notin\omega$,
then $w_2$ would be contained in 
$\bigcap_{\{\omega\in\Phi\, :\, w_2\in \omega\}}\omega$
but not in $\bigcap_{\{\omega\in\Phi\, 
:\, w_1\in \omega\}}\omega$, so that the
two sets will be different.
Finally if ${\rm cone}(\lambda\cup\lambda_1) 
\cap\mathring{\lambda}_2\neq 0$
holds then any orbit cone $\omega\in\Phi(\lambda)$
which contains $\lambda_1$ must intersect
the interior of $\lambda_2$ so that
$\lambda_2\subseteq\omega$.
Thus 
$\{\omega\in\Phi\, :\, w_1\in \omega\}
 \subseteq
 \{\omega\in\Phi\, :\, w_2\in \omega\}$
holds. Similarly one proves the opposite
inclusion.
\end{proof}

\begin{Proposition}\label{prop_5chambers}
Let $X$ be a toric $3$-fold such there are two Mori chambers in $\Mov(X)$ whose union gives a single stable base locus chamber. Then there are at least five Mori chambers inside $\Mov(X)$.
\end{Proposition}
\begin{proof}
Assume there are four Mori chambers inside $\Mov(X)$, and consider the chamber corresponding to $\Nef(X)$. If the three remaining chambers $\mathcal{C}_1,\mathcal{C}_2,\mathcal{C}_3$ are adjacent to $\Nef(X)$, and $C_1,C_2,C_3$ are the corresponding three distinct flipping curves, then by Nakamaye's theorem \cite[Theorem 10.3.5]{La04II} we get that the stable base locus of a divisor in $\mathcal{C}_i$ is $C_i$, and hence the stable base loci of divisors in $\mathcal{C}_1,\mathcal{C}_2,\mathcal{C}_3$ are distinct.
Note that this argument works also
for three or less chambers.

If $\mathcal{C}_1,\mathcal{C}_2$ are adjacent to $\Nef(X)$ and $\mathcal{C}_3$ is not then again Nakamaye's theorem yields that divisors in $\mathcal{C}_1,\mathcal{C}_2$ have as base loci two different irreducible curves. Furthermore, the base locus of a divisor in $\mathcal{C}_3$ is a curve with two components. 

Now, let $\mathcal C_1,\dots,\mathcal C_4$ be
four consecutive maximal Mori chambers 
contained in the moving cone of a 
$\mathbb Q$-factorial toric threefold $X$
of Picard rank $n$ and let $X = X_1,\dots,X_4$
be the corresponding birational models
of $X$.
Assume that $\mathcal C_1$ is the 
ample cone of $X$.
Let $\tau_{ij} := \mathcal C_i\cap \mathcal C_j$
and let $\Gamma_{ij}\subseteq X_j$ be the 
new irreducible curve produced by 
the wall crossing of $\tau_{ij}$.
By the previous argument the two
chambers with the same stable locus 
are $\mathcal C_3$ and $\mathcal C_4$.
In order for $\mathcal C_3$ and 
$\mathcal C_4$ to have the same
stable base locus on $X$ the center
$\Gamma_{34}\subseteq X_4$
of the small $\mathbb{Q}$-factorial modification given by the wall crossing 
of $\tau_{34}$ must be an irreducible 
curve which does not exists 
in $X$.
This is possible only if this curve
is the strict transform of the curve 
$\Gamma_{12}$ created by the wall crossing 
of $\tau_{12}$.
In particular 
\[
 \Gamma_{34}
 \sim
 \alpha \Gamma_{12}+\beta\Gamma_{23}, 
\]
with $\alpha,\beta\in\mathbb Q$.
If we denote by $H_{ij}$ the linear 
span of the cone $\tau_{ij}$ in
${\rm Cl}_{\mathbb Q}(X)$, then
the above discussion implies that 
the three hyperplanes $H_{12},
H_{23}, H_{34}$ lie on a pencil.
Each of these hyperplanes must
contain at least $n-1$ of the $n+3$ 
classes of generators of the Cox ring. 
Assuming that the intersection 
$H_{12}\cap H_{23}\cap H_{34}$ 
contains $r$ of these classes then 
we have the following inequality
$$
 3(n-1-r)+r \leq n+3
$$
which implies $r\geq n-3$.
If $r=n-3$ then each of the three 
hyperplanes must contain exactly 
two more classes of generators.
Up to symmetries the possible configurations
of the six classes on the three hyperplanes 
are the following.
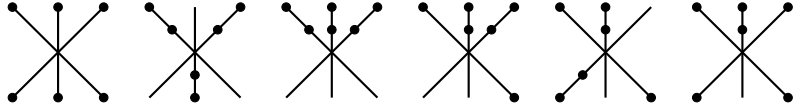
\begin{center}
\begin{tikzpicture}[scale=.3]
\tkzDefPoint(2,2){P1}
\tkzDefPoint(1,1){Q1}
\tkzDefPoint(-2,-2){P2}
\tkzDefPoint(-1,-1){Q2}
\tkzDefPoint(0,2){P3}
\tkzDefPoint(0,1){Q3}
\tkzDefPoint(0,-2){P4}
\tkzDefPoint(0,-1){Q4}
\tkzDefPoint(-2,2){P5}
\tkzDefPoint(-1,1){Q5}
\tkzDefPoint(2,-2){P6}
\tkzDefPoint(1,-1){Q6}
\tkzDrawSegments[thick](P1,P2 P3,P4 P5,P6)
\tkzDrawPoints[fill=black,color=black,size=8](P1,P2,P3,P4,P5,P6)

\begin{scope}[xshift=6cm]
\tkzDefPoint(2,2){P1}
\tkzDefPoint(1,1){Q1}
\tkzDefPoint(-2,-2){P2}
\tkzDefPoint(-1,1){Q2}
\tkzDefPoint(-2,2){R2}
\tkzDefPoint(0,2){P3}
\tkzDefPoint(0,1){Q3}
\tkzDefPoint(0,-2){P4}
\tkzDefPoint(0,-1){Q4}
\tkzDefPoint(-2,2){P5}
\tkzDefPoint(-1,1){Q5}
\tkzDefPoint(2,-2){P6}
\tkzDefPoint(1,-1){Q6}
\tkzDrawSegments[thick](P1,P2 P3,P4 P5,P6)
\tkzDrawPoints[fill=black,color=black,size=8](P1,Q1,R2,Q2,P4,Q4)
\end{scope}

\begin{scope}[xshift=12cm]
\tkzDefPoint(2,2){P1}
\tkzDefPoint(1,1){Q1}
\tkzDefPoint(-2,-2){P2}
\tkzDefPoint(-1,1){Q2}
\tkzDefPoint(-2,2){R2}
\tkzDefPoint(0,2){P3}
\tkzDefPoint(0,1){Q3}
\tkzDefPoint(0,-2){P4}
\tkzDefPoint(0,-1){Q4}
\tkzDefPoint(-2,2){P5}
\tkzDefPoint(-1,1){Q5}
\tkzDefPoint(2,-2){P6}
\tkzDefPoint(1,-1){Q6}
\tkzDrawSegments[thick](P1,P2 P3,P4 P5,P6)
\tkzDrawPoints[fill=black,color=black,size=8](P1,Q1,R2,Q2,P3,Q3)
\end{scope}

\begin{scope}[xshift=18cm]
\tkzDefPoint(2,2){P1}
\tkzDefPoint(1,1){Q1}
\tkzDefPoint(-2,-2){P2}
\tkzDefPoint(-1,1){Q2}
\tkzDefPoint(2,-2){R2}
\tkzDefPoint(0,2){P3}
\tkzDefPoint(0,1){Q3}
\tkzDefPoint(0,-2){P4}
\tkzDefPoint(0,-1){Q4}
\tkzDefPoint(-2,2){P5}
\tkzDefPoint(-1,1){Q5}
\tkzDefPoint(2,-2){P6}
\tkzDefPoint(1,-1){Q6}
\tkzDrawSegments[thick](P1,P2 P3,P4 P5,P6)
\tkzDrawPoints[fill=black,color=black,size=8](P1,Q1,R2,P5,P3,Q3)
\end{scope}

\begin{scope}[xshift=24cm]
\tkzDefPoint(2,2){P1}
\tkzDefPoint(1,1){Q1}
\tkzDefPoint(-2,-2){P2}
\tkzDefPoint(-1,-1){Q2}
\tkzDefPoint(2,-2){R2}
\tkzDefPoint(0,2){P3}
\tkzDefPoint(0,1){Q3}
\tkzDefPoint(0,-2){P4}
\tkzDefPoint(0,-1){Q4}
\tkzDefPoint(-2,2){P5}
\tkzDefPoint(-1,1){Q5}
\tkzDefPoint(2,-2){P6}
\tkzDefPoint(1,-1){Q6}
\tkzDrawSegments[thick](P1,P2 P3,P4 P5,P6)
\tkzDrawPoints[fill=black,color=black,size=8](P2,Q2,R2,P5,P3,Q3)
\end{scope}

\begin{scope}[xshift=30cm]
\tkzDefPoint(2,2){P1}
\tkzDefPoint(1,1){Q1}
\tkzDefPoint(-2,-2){P2}
\tkzDefPoint(-1,1){Q2}
\tkzDefPoint(2,-2){R2}
\tkzDefPoint(0,2){P3}
\tkzDefPoint(0,1){Q3}
\tkzDefPoint(0,-2){P4}
\tkzDefPoint(0,-1){Q4}
\tkzDefPoint(-2,2){P5}
\tkzDefPoint(-1,1){Q5}
\tkzDefPoint(2,-2){P6}
\tkzDefPoint(1,-1){Q6}
\tkzDrawSegments[thick](P1,P2 P3,P4 P5,P6)
\tkzDrawPoints[fill=black,color=black,size=8](P1,P2,R2,P5,P3,Q3)
\end{scope}

\end{tikzpicture}
\end{center}
In the last four cases at least one of the
facets $\tau_{ij}$ would be not contained 
in the interior of the moving cone
(see Proposition~\cite[3.3.2.3]{ADHL15}).
In the first two cases there would be 
at least six Mori chambers in the moving 
cone.

If $r = n-2$ then each of the three
hyperplanes contains at least one 
more class.
The possible configurations are the
following, where in each case still 
two more generators have to be 
added to the picture.
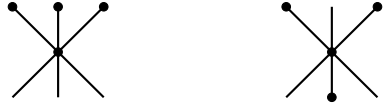
\begin{center}
\begin{tikzpicture}[scale=.3]
\tkzDefPoint(0,0){O}
\tkzDefPoint(2,2){P1}
\tkzDefPoint(-2,-2){P2}
\tkzDefPoint(0,2){P3}
\tkzDefPoint(0,-2){P4}
\tkzDefPoint(-2,2){P5}
\tkzDefPoint(2,-2){P6}
\tkzDrawSegments[thick](P1,P2 P3,P4 P5,P6)
\tkzDrawPoints[fill=black,color=black,size=8](O,P1,P5,P3)


\begin{scope}[xshift=12cm]
\tkzDefPoint(0,0){O}
\tkzDefPoint(2,2){P1}
\tkzDefPoint(-2,-2){P2}
\tkzDefPoint(0,2){P3}
\tkzDefPoint(0,-2){P4}
\tkzDefPoint(-2,2){P5}
\tkzDefPoint(2,-2){P6}
\tkzDrawSegments[thick](P1,P2 P3,P4 P5,P6)
\tkzDrawPoints[fill=black,color=black,size=8](O,P1,P5,P4)
\end{scope}
\end{tikzpicture}
\end{center}
In the first case it is possible to 
fulfill these conditions only if the two
more classes are as in the picture below.
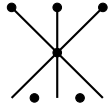
\begin{center}
\begin{tikzpicture}[scale=.3]
\tkzDefPoint(0,0){O}
\tkzDefPoint(2,2){P1}
\tkzDefPoint(-2,-2){P2}
\tkzDefPoint(0,2){P3}
\tkzDefPoint(0,-2){P4}
\tkzDefPoint(-2,2){P5}
\tkzDefPoint(2,-2){P6}
\tkzDefPoint(-1,-2){Q1}
\tkzDefPoint(1,-2){Q2}
\tkzDrawSegments[thick](P1,P2 P3,P4 P5,P6)
\tkzDrawPoints[fill=black,color=black,size=8](O,P1,P5,P3,Q1,Q2)
\end{tikzpicture}
\end{center}
In this case the moving cone contains 
at least five Mori chambers, as shown
in Example~\ref{3toric_not_smooth}.
In the last case, up to symmetry 
there is again one possibility for
the open chambers where the 
two new classes can belong.
This is displayed in the picture below.
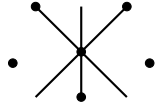
\begin{center}
\begin{tikzpicture}[scale=.3]
\tkzDefPoint(0,0){O}
\tkzDefPoint(2,2){P1}
\tkzDefPoint(-2,-2){P2}
\tkzDefPoint(0,2){P3}
\tkzDefPoint(0,-2){P4}
\tkzDefPoint(-2,2){P5}
\tkzDefPoint(2,-2){P6}
\tkzDefPoint(-3,-0.5){Q1}
\tkzDefPoint(3,-0.5){Q2}
\tkzDrawSegments[thick](P1,P2 P3,P4 P5,P6)
\tkzDrawPoints[fill=black,color=black,size=8](O,P1,P5,P4,Q1,Q2)
\end{tikzpicture}
\end{center}
However, in this last case we can have either five or three chambers in the movable cone.

If $r = n-1$ then again each of 
the three hyperplanes contains at least 
one more class, so that the above configurations 
still apply. Here we have to add only 
one more class to the picture and 
thus one of the three cones $\tau_{ij}$
would not be contained in the interior
of the moving cone.
\end{proof}

Examples \ref{3toric_not_smooth} and \ref{4fold_toric} in the next section show that Proposition \ref{prop_5chambers} is sharp in
the sense that one can has Mori chamber decomposition distinct from the stable base locus decomposition inside the movable cone with five chambers and with three chambers in dimension higher than three.

\section{Examples}\label{examples}
In this section we give examples of varieties for which the Mori chamber and the stable base locus decomposition do not coincide, and we analyze this phenomenon from both the geometrical and the combinatorial point of view. 

\begin{Example}(Birational viewpoint)
Consider a plane $\Pi\subset\mathbb{P}^n$ and five general points $p_1,\dots,p_5\in \Pi$. Let $f:X\rightarrow\mathbb{P}^n$ be the blow-up of $\mathbb{P}^n$ at $p_1,\dots,p_5$ with exceptional divisors $E_1,\dots,E_5$. Then the strict transform $\widetilde{\Pi}\subset X$ of $\Pi$ is a del Pezzo surface of degree four. In particular $\widetilde{\Pi}$ is a Mori dream space. 

Let $e_1,\dots,e_5$ be classes of a line in the exceptional divisors, and $l$ the pull-back a a general line in $\mathbb{P}^n$. Let $\widetilde{C}\subset X$ be an irreducible curve. If $\widetilde{C}$ gets contracted by $f$ the $\widetilde{C}\sim me_i$ with $m>0$ for some $i\in\{1,\dots,5\}$. Otherwise, we may write $\widetilde{C}\sim dl-m_1e_1-\dots-m_5e_5$, that is $\widetilde{C}$ is the strict transform of a curve $C\subset\mathbb{P}^n$ of degree $d$ having multiplicity $m_i$ at $p_i$ for $i=1,\dots,5$. 

If $d<m_1 + \dots + m_5$ then $C\subset \Pi$ and $\widetilde{C}\subset\widetilde{\Pi}$. In this case we may write $\widetilde{C}$ as a linear combination with non-negative coefficients of $e_1,\dots,e_5,l-e_i-e_j,2l-e_1-\dots -e_5$ since these are the generators of $\NE(\widetilde{\Pi})$. If $d\geq m_1+\dots +m_5$ then we may write
$$\widetilde{C}\sim m_1(l-e_1)+\dots m_5(l-e_5)+(d-m_1-\dots -m_5)l$$
where $l-e_i = (l-e_i-e_j)+e_j$. Therefore, 
$$\NE(X) = \left\langle e_i, l-e_i-e_j,2l-e_1-\dots -e_5\right\rangle$$
Now, let $D\subset\mathbb{P}^n$ be the divisor given by the union of $n-2$ general hyperplanes containing $\Pi$. For the strict transform $\widetilde{D}\subset X$ of $D$ we have $\widetilde{D}\sim (n-2)(H-E_1-\dots -E_5)$, and
$$-K_X-\epsilon\widetilde{D}\sim (n+1-\epsilon (n-2))H-((n-1)-\epsilon (n-2))(E_1+\dots E_5)$$
where $H$ is the pull-back of the hyperplane section of $\mathbb{P}^n$ via the blow-up morphism. Now, note that $(-K_X-\epsilon\widetilde{D})\cdot (l-e_i-e_j) = \epsilon (n-2)-n+3$, $(-K_X-\epsilon\widetilde{D})\cdot (2l-e_i-\dots -e_5) = 3\epsilon (n-2)-3n+7$, and $(-K_X-\epsilon\widetilde{D})\cdot e_i = n-1-\epsilon (n-2)$. Therefore, if $n\geq 3$ we have that $-K_X-\epsilon\widetilde{D}$ is ample for any $\frac{3n-7}{3n-6} < \epsilon \frac{n-1}{n-2}$. 

Furthermore, since $\widetilde{D}$ is the union of $n-2$ smooth irreducible divisors intersecting transversally along the surface $\widetilde{\Pi}$ the pair $(X,\epsilon\widetilde{D})$ is klt for any $0<\epsilon < 1$. Then, for any $\frac{3n-7}{3n-6} < \epsilon < 1$ the divisor $\epsilon\widetilde{D}$ induces a log Fano structure on $X$, and \cite[Corollary 1.3.2]{BCHM10} yields that $X$ is a Mori dream space.

Now, consider the divisor $D_i \sim 2H-E_1-\dots -\widehat{E}_i-\dots-E_4$ where the hat means that $E_i$ does no appear in the expression of $D_i$. Note that $D_i$ is nef, and since $D_i^n >0$ it is also big. Therefore, $D_i$ is semi-ample and big. Now, consider a curve $\widetilde{C}\subset X$
$$\widetilde{C}\sim \alpha_{1,2}(l-e_1-e_2)+\dots +\alpha_{4,5}(l-e_4-e_5)+\beta \widetilde{C}_{1,\dots,5}+\gamma_1e_1+\dots +\gamma_5e_5$$
where $\widetilde{C}_{1,\dots,5}\sim 2l-e_1-\dots -e_5$. Assume that $D_i\cdot \widetilde{C} =\gamma_1+\dots +\gamma_5 = 0$. Hence we may write
$$\widetilde{C}\sim (2\beta+\alpha_{1,2}+\dots+\alpha_{4,5})l-(\beta+\alpha_{1,2}+\dots+\alpha_{1,5})e_1 -\dots -(\beta+\alpha_{4,5})e_5$$
Since  $2\beta+\alpha_{1,2}+\dots+\alpha_{4,5} < \beta+\alpha_{1,2}+\dots+\alpha_{1,5}+\dots +\beta+\alpha_{4,5}$ we conclude that $\widetilde{C}\subset\widetilde{\Pi}$. 

Then, a large enough multiple of $D_i$ induces a birational morphism $f_{D_i}:X\rightarrow X_i$ contracting $\widetilde{\Pi}$ onto a $\mathbb{P}^1$ contained in $X_i$, and whose exceptional locus is exactly $\widetilde{\Pi}$, that is $\Exc(f_i) = \widetilde{\Pi}$. Indeed, $D_{i|\widetilde{\Pi}}$ yields the fibration $\widetilde{\Pi}\rightarrow \mathbb{P}^1$ induced by the linear system of conics in $\Pi$ through $p_1,\dots,\widehat{p}_i,\dots,p_5$.

Let $f_{D_i}:X\rightarrow X_i$ and $f_{D_j}:X\rightarrow X_j$ be the morphisms induced respectively by $D_i$ and $D_j$. From now on we will assume that $n\geq 4$ so that $f_{D_i}$ is a small contraction. Then $\Exc(f_{D_i}) = \Exc(f_{D_j}) = \widetilde{\Pi}$. On the other hand, $D_i$ and $D_j$ give rise to two different flops
  \[
  \begin{tikzpicture}[xscale=1.9,yscale=-1.2]
    \node (A0_0) at (0, 0) {$X^{-}$};
    \node (A0_2) at (2, 0) {$X$};
    \node (A0_4) at (4, 0) {$X^{+}$};
    \node (A1_1) at (1, 1) {$X_i$};
    \node (A1_3) at (3, 1) {$X_j$};
    \path (A0_2) edge [->,dashed,swap]node [auto] {$\scriptstyle{\chi^{-}}$} (A0_0);
    \path (A0_2) edge [->]node [auto] {$\scriptstyle{f_{D_i}}$} (A1_1);
    \path (A0_0) edge [->]node [auto] {$\scriptstyle{}$} (A1_1);
    \path (A0_4) edge [->]node [auto] {$\scriptstyle{}$} (A1_3);
    \path (A0_2) edge [->,dashed]node [auto] {$\scriptstyle{\chi^{+}}$} (A0_4);
    \path (A0_2) edge [->,swap]node [auto] {$\scriptstyle{f_{D_j}}$} (A1_3);
  \end{tikzpicture}
  \]
Therefore, $X^{+}$ and $X^{-}$ correspond to two different chambers $\mathcal{C}^{+}$ and $\mathcal{C}^{-}$ of the Mori chamber decomposition of $\Mov(X)$. On the other hand, by Nakamaye's theorem \cite[Theorem 10.3.5]{La04II} the base locus of $\chi^{+}$ is $\Exc(f_{D_j})$ and the base locus of $\chi^{-}$ is $\Exc(f_{D_i})$. These base loci are respectively the stable base loci of $D_i$ and $D_j$. 

Finally, since $\Exc(f_{D_i}) = \Exc(f_{D_j}) = \widetilde{\Pi}$ we conclude that $\mathcal{C}^{-}\cup \mathcal{C}^{+}$ is a unique chamber of the stable base locus decomposition of $\Mov(X)$.

More generally, for $i = 1,\dots,5$ we get five different chambers, all of them adjacent to $\Nef(X)$, of the Mori chamber decomposition of $\Mov(X)$ whose union gives a single chamber of its stable base locus decomposition. Furthermore, the stable base locus of a divisor in this chamber is exactly the surface $\widetilde{\Pi}$. 
\end{Example}

\begin{Example}(Cox rings viewpoint)\label{ex_nc}
Let $X$ be the toric variety with Cox ring
$\mathcal{R}(X) := K[T_1,\dots,T_5]$ whose 
grading matrix and irrelevant ideal are
the following
\[
  Q =
  \begin{bmatrix}
   1&0&-2&2&-1\\
   0&1&-1&1&-1
  \end{bmatrix}
  \qquad
  \qquad
  \mathcal J_{\rm irr}(X) 
  = 
  \langle T_1,T_4\rangle
  \cap
  \langle T_1,T_5\rangle
  \cap
  \langle T_3,T_5\rangle
\]
The degrees of the generators are displayed
in the following picture together with the
semi-ample cone $\lambda_A$ which is the 
gray region
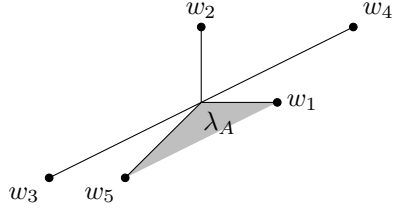
\begin{center}
 \begin{tikzpicture}
   \path [fill=lightgray] (0,0) -- (1,0) -- (-1,-1);
  \foreach \x/\y in {1/0/1,0/1,-2/-1,2/1,-1/-1}
   {\draw[fill] (\x,\y) circle [radius=0.05];
    \draw[-] (0,0) -- (\x,\y);}
   \node[right] at (1,0) {$w_1$};
   \node[above] at (0,1) {$w_2$};
   \node[below left] at (-2,-1) {$w_3$};
   \node[above right] at (2,1) {$w_4$};
   \node[below left] at (-1,-1) {$w_5$};
   \node[below right] at (-0.1,0) {$\lambda_A$};
 \end{tikzpicture}
\end{center}
The matrix $Q$ defines the following exact 
sequence where the right hand side
$\mathbb Z^2$ is identified with the 
divisor class group of $X$
$$
0\rightarrow \mathbb Z^3\rightarrow \mathbb Z^5\xrightarrow{Q} \mathbb Z^2\rightarrow 0
$$
Denote by $\gamma$ the positive quadrant of
$\mathbb Q^5$.
Since the Cox ring is a polynomial ring 
the $\mathfrak F$-faces are all the
faces of $\gamma$ and the orbit cones
are all their projections in $\mathbb Q^2$.
It follows that the maximal GIT chambers of 
$X$ are all the $2$-dimensional cones generated 
by pairs of consecutive rays.
Recall that $\lambda_A$ is the GIT chamber
corresponding to the semi-ample cone, 
generated by $w_1,w_5$. The corresponding 
bunch of orbit cones is 
\[
 \Phi(\lambda_A)
 :=
 \{\text{orbit cones $\omega$ such that 
 $\omega^\circ\supseteq\lambda_A^\circ$}\}
\]
In particular the only orbit cone of $\Phi$
which contains $w_2$ is the whole of $\mathbb Q^2$. Denote by $\lambda_{i,j}$ the cone determined by $\omega_i$ and $\omega_j$, and observe that ${\rm cone}(\lambda_A\cup\lambda_{2,3})
= {\rm cone}(\lambda_A\cup\lambda_{2,4}) = 
\mathbb Q^2$. Thus, by Proposition~\ref{re:sbl}, we conclude that $\lambda^{\rm sbl}(w_2) = \lambda_{2,3}\cup\lambda_{2,4}$.
\end{Example}

\begin{Example}(Fans viewpoint)\label{ex2}
Let $v_1\dots,v_5\in\mathbb Z^3$ be a set
of vectors which is Gale dual to the set
$w_1,\dots,w_5\in \mathbb Z^2$ in Example \ref{ex_nc}. We can assume $v_1\dots,v_5$
to be the five columns of the following matrix
\[
 \begin{bmatrix}
  1&0&0&-1&-1\\
  0&1&0&1&2\\
  0&0&1&1&0
 \end{bmatrix}
\]
These vectors generate the one dimensional
cones of a fan $\Sigma$ whose cones are 
displayed in the following picture
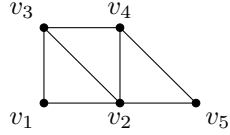
\begin{center}
 \begin{tikzpicture}
  \draw[-] (0,0) -- (2,0) -- (1,1) -- (0,1) -- (0,0);
  \draw[-] (0,1) -- (1,0) -- (1,1);
  \foreach \x/\y in {0/0,1/0,0/1,1/1,2/0}
   {\draw[fill] (\x,\y) circle [radius=0.05];}
   \node[below left] at (0,0) {$v_1$};
   \node[below] at (1,0) {$v_2$};
   \node[above left] at (0,1) {$v_3$};
   \node[above] at (1,1) {$v_4$};
   \node[below right] at (2,0) {$v_5$};
 \end{tikzpicture}
\end{center}
The toric variety $X = X(\Sigma)$ is the 
same as the one previously defined in Example \ref{ex_nc}.
The displayed fan structure correspond
to choosing the semi-ample chamber 
to be the GIT chamber $\lambda_A$ of
the previous picture.
The following are the fan structures of
the toric varieties whose semi-ample cone
is respectively ${\rm cone}(w_1,w_4)$
and ${\rm cone}(w_3,w_5).$
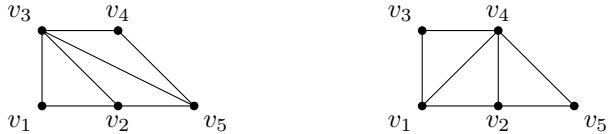
\begin{center}
 \begin{tikzpicture}
  \draw[-] (0,0) -- (2,0) -- (1,1) -- (0,1) -- (0,0);
  \draw[-] (0,1) -- (1,0);
  \draw[-] (0,1) -- (2,0);
  \foreach \x/\y in {0/0,1/0,0/1,1/1,2/0}
   {\draw[fill] (\x,\y) circle [radius=0.05];}
   \node[below left] at (0,0) {$v_1$};
   \node[below] at (1,0) {$v_2$};
   \node[above left] at (0,1) {$v_3$};
   \node[above] at (1,1) {$v_4$};
   \node[below right] at (2,0) {$v_5$};
 \begin{scope}[xshift=5cm]
  \draw[-] (0,0) -- (2,0) -- (1,1) -- (0,1) -- (0,0);
  \draw[-] (1,0) -- (1,1);
  \draw[-] (0,0) -- (1,1);
  \foreach \x/\y in {0/0,1/0,0/1,1/1,2/0}
   {\draw[fill] (\x,\y) circle [radius=0.05];}
   \node[below left] at (0,0) {$v_1$};
   \node[below] at (1,0) {$v_2$};
   \node[above left] at (0,1) {$v_3$};
   \node[above] at (1,1) {$v_4$};
   \node[below right] at (2,0) {$v_5$};
 \end{scope}
 \end{tikzpicture}
\end{center}
\end{Example}

\begin{Example}(Compactification)
Let $Y$ be the toric variety with Cox ring
$\mathcal{R} := K[T_1,\dots,T_6]$ whose 
grading matrix and irrelevant ideal are
the following
\begin{align*}
  Q =
  \begin{bmatrix}
    0&2&2&0&1&1\\
    0&3&1&1&0&1\\
    1&0&2&0&2&1
  \end{bmatrix}
  \qquad
  \qquad
  \mathcal J_{\rm irr}(Y) 
  = & 
  \langle T_1,T_4\rangle
  \cap
  \langle T_1,T_5\rangle
  \cap
  \langle T_3,T_5\rangle\\
  & \cap
  \langle T_2,T_3,T_6\rangle
  \cap
  \langle T_2,T_4,T_6\rangle
\end{align*}
The toric variety $Y$ is a completion of the variety
$X$ in Example \ref{ex_nc} obtained by adding the vector 
$(1,-4,-2)$ to the primitive generators of the
one dimensional cones of $X$. The maximal 
cones of $Y$ have the following indexes:
$[ 1, 2, 3 ]$, $[ 2, 3, 4 ]$, $[ 2, 4, 5 ]$, 
$[ 1, 2, 6 ]$, $[ 2, 5, 6 ]$, $[ 1, 3, 6 ]$, 
$[ 3, 4, 6 ]$, $[ 4, 5, 6 ]$.
The variety $Y$ is $\mathbb Q$-factorial,
non-Gorenstein, with $2K_Y$ Cartier.
The effective cone of $Y$ has $16$ maximal 
GIT chambers, and the moving cone of $Y$ has 
$3$ maximal chambers. The columns of each 
of the following matrices generate a maximal 
GIT chamber of the moving cone, where the first one
corresponds to the semi-ample cone
$$
 \mathcal A =
  \begin{bmatrix}
   1&1&2&2\\
   1&1&1&1\\
   1&2&3&4
  \end{bmatrix}
  \qquad
 \mathcal M_1 =
  \begin{bmatrix}
   1&1&2&2\\
   1&1&3&3\\
   1&2&2&4
  \end{bmatrix}
  \qquad
 \mathcal M_2 =
  \begin{bmatrix}
   1&2&4&6\\
   1&1&3&3\\
   1&3&4&8
  \end{bmatrix}
$$
The columns of each of the following matrices
generate a maximal GIT chamber of the effective
cone of $X$.
$$
 \mathcal C_1 =
  \begin{bmatrix}
    1&2&4\\
    1&3&3\\
    1&0&4
  \end{bmatrix}
  \qquad
 \mathcal C_2 =
  \begin{bmatrix}
    1&2&2\\
    1&3&3\\
    1&0&2
  \end{bmatrix}
$$
The stable base locus of a divisor which lies in
the interior of either $\mathcal C_1$ or 
$\mathcal C_2$ is $V(T_2)$. Thus the union of these 
two chambers is a unique stable base locus chamber.
On the other hand, if we denote by $X_i$ the
projective toric variety whose semi-ample cone 
is given by $\mathcal M_i$, then from the point of view of $X_i$ the GIT chambers
coincide with the stable base locus chambers.
\end{Example}

\begin{Example}
Consider the toric variety $Z=Z(\Lambda)$ with fan $\Lambda\subset \mathbb{Q}^3$ given in the following picture
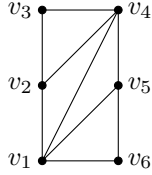
\begin{center}
 \begin{tikzpicture}
  \draw[-] (0,0) -- (1,0) -- (1,1) -- (1,2) -- (0,2) -- (0,1) -- (0,0);
  \draw[-] (1,1) -- (0,0) -- (1,2) -- (0,1);
  \foreach \x/\y in {0/0,1/0,1/1,1/2,0/1,0/2}
   {\draw[fill] (\x,\y) circle [radius=0.05];}
   \node[left] at (0,0) {$v_1$};
   \node[left] at (0,1) {$v_2$};
   \node[left] at (0,2) {$v_3$};
   \node[right] at (1,2) {$v_4$};
   \node[right] at (1,1) {$v_5$};
   \node[right] at (1,0) {$v_6$};
 \end{tikzpicture}
\end{center}
One can take for instance $v_1,\dots, v_6$ to be the columns of the following matrix
\begin{center}
$\left[\begin{array}{rrrrrrr}
-1 & -1 & -1 & 1 & 1 & 1\\ 
-1 & 0 & 1 & -1 & 0 & 1\\ 
1 & 1 & 1 & 1 & 1 & 1 
\end{array}\right]$
\end{center}
Then $Z$ is a non-complete Mori dream space $3$-fold with Picard number three and with Cox ring generated by six free variables with degrees given by the columns $w_1,\dots,w_6$ of the following matrix
\begin{center}
$\left[\begin{array}{rrrrrrr}
1 & -2 & 1 & 0 & 0 & 0 \\ 
1 & -1 & 0 & -1 & 1 & 0 \\ 
2 & -2 & 0 & -1 & 0 & 1
\end{array}\right]$
\end{center}
The effective cone of $Z$ is the whole of $\mathbb{Q}^3$, it has $14$ maximal GIT chambers and six of them are in the movable cone. Let us denote by $Z_j:=Z(\lambda_j),j=1,\dots, 6$ the different models where $Z=Z_1,$ and $\lambda_j,j=1,\dots,6$ are the corresponding GIT chambers. The related fans are the following
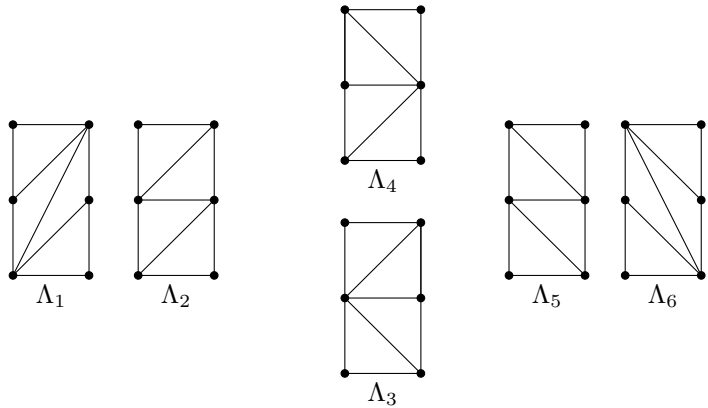
\begin{center}
 \begin{minipage}[b]{0.25\textwidth}
 \begin{tikzpicture}
  \draw[-] (0,0) -- (1,0) -- (1,1) -- (1,2) -- (0,2) -- (0,1) -- (0,0);
  \draw[-] (1,1) -- (0,0) -- (1,2) -- (0,1);
    \foreach \x/\y in {0/0,1/0,1/1,1/2,0/1,0/2}
   {\draw[fill] (\x,\y) circle [radius=0.05];}
    \node[below] at (0.5,0) {$\Lambda_1$}; 
 \end{tikzpicture} 
   \hspace{0.3cm}
  \begin{tikzpicture}
  \draw[-] (0,0) -- (1,0) -- (1,1) -- (1,2) -- (0,2) -- (0,1) -- (0,0);
  \draw[-] (0,0) -- (1,1) -- (0,1) -- (1,2);
    \foreach \x/\y in {0/0,1/0,1/1,1/2,0/1,0/2}
   {\draw[fill] (\x,\y) circle [radius=0.05];}
    \node[below] at (0.5,0) {$\Lambda_2$};  
 \end{tikzpicture}\vspace{1.3cm}\end{minipage}
 \begin{minipage}[b]{0.12\textwidth}
  \begin{tikzpicture}
  \draw[-] (0,0) -- (1,0) -- (1,1) -- (1,2) -- (0,2) -- (0,1) -- (0,0);
  \draw[-] (0,0) -- (1,1) -- (0,2) -- (0,1) -- (1,1);
    \foreach \x/\y in {0/0,1/0,1/1,1/2,0/1,0/2}
   {\draw[fill] (\x,\y) circle [radius=0.05];}
    \node[below] at (0.5,0) {$\Lambda_4$};  
 \end{tikzpicture}\vspace{0.2cm}
 \begin{tikzpicture}
  \draw[-] (0,0) -- (1,0) -- (1,1) -- (1,2) -- (0,2) -- (0,1) -- (0,0);
 \draw[-] (1,0) -- (0,1) -- (1,2) -- (1,1) -- (0,1);
    \foreach \x/\y in {0/0,1/0,1/1,1/2,0/1,0/2}
   {\draw[fill] (\x,\y) circle [radius=0.05];}
    \node[below] at (0.5,0) {$\Lambda_3$};  
\end{tikzpicture} \end{minipage}
 \begin{minipage}[b]{0.25\textwidth}
  \begin{tikzpicture}
  \draw[-] (0,0) -- (1,0) -- (1,1) -- (1,2) -- (0,2) -- (0,1) -- (0,0);
\draw[-] (1,0) -- (0,1) -- (1,1) -- (0,2);
    \foreach \x/\y in {0/0,1/0,1/1,1/2,0/1,0/2}
   {\draw[fill] (\x,\y) circle [radius=0.05];}
    \node[below] at (0.5,0) {$\Lambda_5$};  
 \end{tikzpicture} \hspace{0.3cm}\begin{tikzpicture}
  \draw[-] (0,0) -- (1,0) -- (1,1) -- (1,2) -- (0,2) -- (0,1) -- (0,0);
\draw[-] (0,1) -- (1,0) -- (0,2) -- (1,1);
    \foreach \x/\y in {0/0,1/0,1/1,1/2,0/1,0/2}
   {\draw[fill] (\x,\y) circle [radius=0.05];} 
    \node[below] at (0.5,0) {$\Lambda_6$}; 
 \end{tikzpicture}\vspace{1.3cm}\end{minipage}
\end{center}
and the possible flips are as follows
$$
  
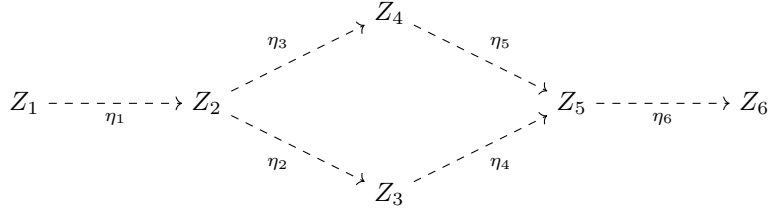
\begin{tikzpicture}[xscale=1.2,yscale=-1.2]
    \node (A0_1) at (0, 0) {$Z_1$};
    \node (A0_2) at (2, 0) {$Z_2$};
    
    \node (A0_3) at (4, 1) {$Z_3$};
    
    \node (A0_4) at (4, -1) {$Z_4$};
    
    \node (A0_5) at (6, 0) {$Z_5$};
    
    \node (A0_6) at (8, 0) {$Z_6$};
    \path (A0_1) edge [->,dashed,swap]node [auto] {$\scriptstyle{\eta_1}$} (A0_2);
        \path (A0_2) edge [->,dashed,swap]node [auto] {$\scriptstyle{\eta_2}$} (A0_3);
        
        \path (A0_2) edge [->,dashed]node [auto] {$\scriptstyle{\eta_3}$} (A0_4);
            \path (A0_3) edge [->,dashed,swap]node [auto] {$\scriptstyle{\eta_4}$} (A0_5);
                \path (A0_4) edge [->,dashed]node [auto] {$\scriptstyle{\eta_5}$} (A0_5);
                    \path (A0_5) edge [->,dashed,swap]node [auto] {$\scriptstyle{\eta_6}$} (A0_6);
  \end{tikzpicture}
$$
In $Z$ we have that $\lambda_2,\lambda_3,\lambda_4$ are in distinct stable base locus chambers but $\lambda_5$ and $\lambda_6$ are inside the same stable base locus chamber. More precisely, if $w_j\in \lambda_j^\circ,j=1,\dots, 6$ then the stable base loci are
$$\bold{B}(w_1)=\emptyset,
\bold{B}(w_2)=l_{1,4},
\bold{B}(w_3)=l_{1,4}\cup l_{1,5},
\bold{B}(w_4)=l_{1,4}\cup l_{2,4},
\bold{B}(w_5)=\bold{B}(w_6)=l_{1,4}\cup l_{1,5}\cup l_{2,4}
$$
where $l_{i,j}$ is the curve in the intersection of the toric divisors $D_i,D_j$ corresponding to the vectors $v_i,v_j$ in the fan $\Lambda$.

Geometrically we see that the flip $\eta_6$ has $l_{2,5}$ as flipping curve, and $l_{2,5}$ does not exist in $Z_1$. Therefore, the flipping curve in the last flip is not visible from the point of view of $Z_1$.

Note that we can produce an example of a complete $3$-fold $W$ with Picard number four such that $\MCD(W)\neq\SBLD(W)$ taking
$$v_8=\frac{-v_1-\dots -v_7}{6}=(0,0,-1)$$
and including the relevant additional cones.
\end{Example}

The following will be the leading example in the Section \ref{mainSec}. Indeed, we will produce a complete Mori dream space $X$ of Picard rank two such that $\MCD(X)\neq \SBLD(X)$.

\begin{Example}(Fundamental example)\label{ex1}
Let $Z$ be the toric variety with Cox ring
$K[T_1,\dots,T_{11}]$ whose 
grading matrix and irrelevant ideal are
the following
\[
 Q = \begin{bmatrix}
  1&1&2&2&2&2&1&0&0&0&0\\
  0&0&1&1&1&1&2&1&1&1&1
 \end{bmatrix}
 \qquad
  \mathcal J_{\rm irr}(Z) 
  = 
  \langle T_1,T_2\rangle
  \cap
  \langle T_3,\dots,T_{11}\rangle
\]
Denote by $w_i\in\mathbb Z^2$ the 
degree of $T_i$. The following
picture displays the degrees of the
generators of the Cox ring together 
with the three maximal GIT chambers 
of the moving cone, where the shaded one is 
the ample cone
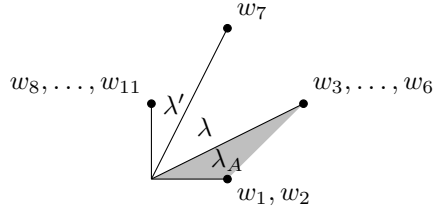
\begin{center}
 \begin{tikzpicture}
   \path [fill=lightgray] (0,0) -- (1,0) -- (2,1);
  \foreach \x/\y in {1/0/1,0/1,2/1,1/2}
   {\draw[fill] (\x,\y) circle [radius=0.05];
    \draw[-] (0,0) -- (\x,\y);}
   \node[below right] at (1,0) {$w_1,w_2$};
   \node[above left] at (0,1) {$w_8,\dots,w_{11}$};
   \node[above right] at (2,1) {$w_3,\dots,w_6$};
   \node[above right] at (1,2) {$w_7$};
   \node at (0.3,1) {$\lambda'$};
   \node at (0.7,0.7) {$\lambda$};
   \node at (1,0.25) {$\lambda_A$};
 \end{tikzpicture}
\end{center}
Let $\overline Z = K^{11}$ be the
spectrum of the Cox ring of $Z$ and let
$\overline X$ be the affine subvariety 
defined by
$$\overline{X} = \{F = G = 0\}\subset\overline{Z}$$
where $F,G$ are general polynomial of degree $(2,2)$ in the $T_i$. that is general linear combinations of the following monomials
\stepcounter{thm}
\begin{equation}\label{mon}
\left\lbrace\begin{array}{llllllll}
T_1^2T_8^2 & T_1^2T_9^2 & T_1^2T_{10}^2 & T_1^2T_{11}^2 & T_2^2T_8^2 & T_2^2T_9^2 & T_2^2T_{10}^2 & T_2^2T_{11}^2;\\ 
T_3T_8 & T_3T_9 & T_3T_{10} & T_3T_{11} & T_4T_8 & T_4T_9 & T_4T_{10} & T_4T_{11};\\ 
T_5T_8 & T_5T_9 & T_5T_{10} & T_5T_{11} & T_6T_8 & T_6T_9 & T_6T_{10} & T_6T_{11};\\ 
T_1T_7 & T_2T_7. & & & & & &  
\end{array}\right. 
\end{equation} 
It is immediate to check that for $F,G$ general enough $\overline{X}$ is irreducible and $\codim_{\overline{X}}(\Sing(\overline{X})) \geq 2$. Since a local complete intersection is Cohen-Macaulay by Serre's criterion on normality we get that $\overline{X}$ is a normal variety. Let $p_Z\colon\widehat Z\to Z$ 
be the characteristic space morphism
of $Z$ and let $\widehat X := \overline X \cap
\widehat Z$. The image of $\widehat X$
via $p_Z$ is a subvariety $X$ of $Z$. Since $\overline{X}$ is irreducible and normal, and $X$ is a GIT quotient of $\overline{X}$ by a reductive group \cite[Theorem 1.24 (vi)]{Br10} yields that $X$ is irreducible and normal as well. We claim that
\[
 \widehat X
 =
 \text{ Zariski closure of $p_Z^{-1}(X\cap Z')$
 in $\widehat Z$}
\]
where $Z'$ is the smooth locus of $Z$.
Indeed $Z'$ contains the open subset $Z''$
of $Z$ obtained by removing the union of all 
the toric subvarieties of the form $p_Z(V(T_i,T_j))$,
for any pair of indexes $i,j$. Since the Zariski
closure of $p_Z^{-1}(X\cap Z'')$ in $\widehat Z$
equals $\widehat X$ the claim follows. 

Observe that $\codim_{\overline{X}}(\overline X\setminus\widehat X)\geq 2$. Thus if one can show that the pull-back
$\imath^*\colon {\rm Cl}(Z)\to{\rm Cl}(X)$,
induced by the inclusion, is an isomorphism
then by \cite[Corollary 4.1.1.5]{ADHL15}, it
follows that the Cox ring of $X$ is
$$
\mathcal{R}(X) = \frac{K[T_1,\dots,T_{11}]}{I(\overline X)}
$$
Note that each of $w_{1..6}$, $w_{3..7}$
and $w_{7..11}$ is an orbit cone of $\overline{X}$.
In particular, the three maximal chambers $\lambda'$,
$\lambda$, $\lambda_A$ of the 
moving cone are GIT chambers. On the other
hand, since $w_{1,2,7}$ is not an orbit cone,
it follows that each orbit cone contains 
\[
  \{\omega\, :\, \lambda_A\subseteq\omega
  \text{ and }\lambda'\subsetneqq\omega\}
 =
  \{\omega\, :\, \lambda_A\subseteq\omega
  \text{ and }\lambda\subsetneqq\omega\}
\]
Then $\lambda$ and $\lambda'$
are contained in the same SBL chamber.
It remains to show that
\[
 \imath^*\colon {\rm Cl}(Z)\to{\rm Cl}(X)
\]
is an isomorphism. Note that the $K^{*}\times K^{*}$ action on $\overline{X}\setminus (V(T_2)\cup V(T_8))$ is trivial, so if we remove the images of $V(T_2)\cup V(T_8)$ 
from $X$, the resulting variety is isomorphic to
an affine space. Therefore, ${\rm Cl}(X)$ is
generated by the classes of the images 
of the two irreducible divisors $V(T_i)\cap \overline X$, with $i\in\{2,8\}$, and $\rho(X)\leq 2$. 

Note that crossing the wall corresponding to $w_1,w_2$ we get a morphism $f:Z\rightarrow\mathbb{P}^1$. Furthermore, $X\subset Z$ is not contained in any fiber of $f$, and hence $f$ restricts to a surjective morphism $f_{|X}:X\rightarrow\mathbb{P}^1$. This forces $\rho(X)\geq 2$. Finally, we conclude that the images of $V(T_2),V(T_8)$ form a basis of $\Cl(X)$ and $\rho(X)=2$.
\end{Example}

We finish the present section with two examples of toric complete varieties in which the Mori chamber and stable base locus decomposition do not coincide even inside the movable cone.
 
\begin{Example}\label{3toric_not_smooth}
Let $X$ be a toric variety with the following grading matrix
\[
 Q := 
 \begin{bmatrix}
    2&0&1&3&1&3\\
    3&2&2&1&1&2\\
    3&1&1&3&3&0
 \end{bmatrix} 
\]
This is a toric $3$-fold whose Mori chamber and stable base locus decomposition do not coincide even inside the movable cone. More precisely, its effective cone has $17$ GIT chambers, of these $5$ are inside the movable cone. The picture below shows a section of the Mori chamber decomposition of $\Eff(X).$ The $5$ inner chambers, two triangles and three quadrilaterals, are the GIT chambers of the movable cone. There are five possible models, the picture shows in black a chosen semi-ample cone, and in the corresponding model the two gray GIT chambers $\mathcal{C}$ and $\mathcal{C}'$ share the stable base locus.
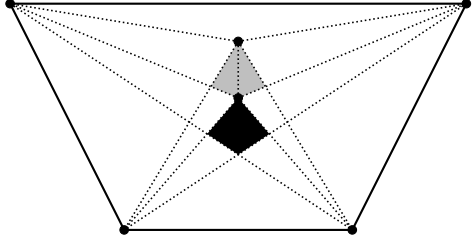
\begin{center}
\begin{tikzpicture}[scale=.5]
\tkzDefPoint(-3,0){P1}
\tkzDefPoint(3,0){P2}
\tkzDefPoint(6,6){P3}
\tkzDefPoint(-6,6){P4}
\tkzDefPoint(0,5){P5}
\tkzDefPoint(0,3.5){P6}
\tkzInterLL(P1,P5)(P2,P4) \tkzGetPoint{Q1}
\tkzInterLL(P2,P5)(P1,P3) \tkzGetPoint{Q2}
\tkzInterLL(P2,P4)(P1,P3) \tkzGetPoint{Q3}
\tkzInterLL(P1,P6)(P2,P4) \tkzGetPoint{Q4}
\tkzInterLL(P2,P6)(P1,P3) \tkzGetPoint{Q5}
\tkzInterLL(P4,P6)(P1,P5) \tkzGetPoint{Q6}
\tkzInterLL(P3,P6)(P2,P5) \tkzGetPoint{Q7}
\tkzFillPolygon[color = black](P6,Q4,Q3,Q5)
\tkzFillPolygon[color = gray!50](P5,P6,Q6)
\tkzFillPolygon[color = gray!50](P5,P6,Q7)
\tkzDrawSegments[thick](P1,P2 P2,P3 P3,P4 P4,P1)
\tkzDrawSegments[densely dotted](P5,P1 P5,P2 P5,P3 P5,P4 P6,P1 P6,P2 P6,P3 P6,P4 P5,P6 P4,P2 P3,P1)
\tkzDrawPoints[fill=black,color=black,size=8](P1,P2,P3,P4,P5,P6)
\end{tikzpicture}
\end{center}
\end{Example}
If one chooses any of the other four models the situation is similar, the two GIT chambers opposite to the chosen semi-ample  cone inside the movable cone will have the same stable base locus.
 
\begin{Example}\label{4fold_toric}
Let $X$ be a toric variety with the following grading matrix
\[
 Q := 
 \begin{bmatrix}
    1&1&0&0&0&0&1\\
    0&0&1&1&0&0&1\\
    0&0&0&0&1&1&1
 \end{bmatrix} 
\]
This is a toric $4$-fold whose Mori chamber and stable base locus decomposition do not coincide inside the movable cone. More precisely, its effective cone has $3$ GIT chambers, all of them inside the movable cone since $\Eff(X)=\Mov(X).$ The picture below shows a section of the Mori chamber decomposition of $\Eff(X).$ There are three possible models, the picture shows in black a chosen semi-ample cone, and in the corresponding model the two gray GIT chambers $\mathcal{C}$ and $\mathcal{C}'$ have the stable base locus. If one chooses any of the other two models the situation is the same.
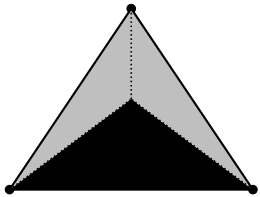
\begin{center}
\begin{tikzpicture}[scale=.8]
\tkzDefPoint(0,0){R1}
\tkzDefPoint(4,0){R2}
\tkzDefPoint(2,3){R3}
\tkzDefPoint(2,1.5){R4}
\tkzDrawPoints[fill=black,color=black,size=8](R1,R2,R3,R4)
\tkzFillPolygon[color = black](R1,R2,R4)
\tkzFillPolygon[color = gray!50](R1,R3,R4)
\tkzFillPolygon[color = gray!50](R2,R3,R4)
\tkzDrawSegments[thick](R1,R2 R2,R3 R3,R1)
\tkzDrawSegments[densely dotted](R1,R4 R2,R4 R3,R4)
\end{tikzpicture}
\end{center}
\end{Example}

\begin{Remark}
All the computations of this section have been implemented in 
Magma~\cite{Magma} and Maple \cite{MDSpackage} programs. For convenience
of the reader we include, as ancillary files in the arXiv version of the paper, the following files:
\begin{enumerate}
\item[-] {\tt Readme.txt}, a text on how to use the remaining files;
\item[-] {\tt SBLib.m}, the Magma library containing all the functions needed to verify our examples;
\item[-] {\tt Examples.txt}, the examples.
\end{enumerate}
For an optimized version, implemented in Maple and Singular, of some of the algorithms presented in this library see \cite{K12}. The library {\tt SBLib.m} contains nine commands
which we briefly describe here.
\begin{itemize}
\item[-] {\tt Ffaces}: computes the $\mathfrak F$-faces of an ideal $I$ of a polynomial ring according
to \cite[Remark 3.1.1.11]{ADHL15}.
\item[-] {\tt Eff}: computes the effective cone of a family of vectors in a rational vector space 
according to \cite[Definition 2.2.2.5]{ADHL15}. It takes as input the grading matrix whose columns are the relevant vectors.
\item[-] {\tt Mov}: computes the moving cone of a family of vectors in a rational vector space according to \cite[Definition 2.2.2.5]{ADHL15}. It takes as input the grading matrix whose columns are the named vectors.
\item[-] {\tt OrbitCones}: computes the orbit cones as projections of the $\mathfrak F$-faces according to \cite[Proposition 3.1.1.10]{ADHL15}. It takes as input a pair consisting of the set of $\mathfrak F$-faces together with the grading matrix.
\item[-] {\tt GitChamber}: computes the GIT chamber defined by a class $w$ according to \cite[Definition 3.1.2.6]{ADHL15}. It takes as input a pair consisting of the set of orbit cones together with a class $w$.
\item[-] {\tt GitFan}: computes the GIT quasi-fan, that is the collection of all the git cones, of the set of orbit cones according to \cite[Algorithm 8]{K12}. It takes as input a pair consisting of the set of 
orbit cones together with a class $w$.
\item[-] {\tt BunchCones}: computes a bunch of orbit cones defined by a GIT chamber $\lambda$ according to \cite[Example 3.1.3.6]{ADHL15}. It takes as input a pair consisting of the set of orbit cones together with a class $w$ in the relative interior of $\lambda$.
\item[-] {\tt SameSbl}: decides whether two classes $w_1,w_2$ have the same stable base locus according
to Proposition \ref{sbl}. It takes as input a triple consisting of a bunch of orbit cones together with the two classes $w_1$ and $w_2$.
\item[-] {\tt FindTriples}: determines all the triples $(\lambda_A,\lambda_1,\lambda_2)$ of GIT chambers such that $\lambda_1$ and $\lambda_2$ are contained in the same stable base locus chamber of the variety whose semi-ample chamber is $\lambda_A$. It works according to Proposition \ref{sbl}, and takes as input a triple consisting of the grading matrix, the set of orbit cones and the GIT fan.
\end{itemize}
\end{Remark}

\section{Smooth toric 3-folds of Picard rank three}\label{sec_toric3fold}
In this section we prove Theorem \ref{prop3foldpic3}. Let $N$ be a finitely generated free 
abelian group, $M := {\rm Hom}(N,\mathbb Z)$, $N_{\mathbb Q} := N\otimes_{\mathbb Z}\mathbb Q$, and let $X := X(\Sigma)$ be a smooth 
projective toric variety defined by a 
fan $\Sigma\subseteq N_{\mathbb Q}$.
Since $X$ is projective $\Sigma$
is the normal fan of the Riemann-Roch
polytope $\Delta\subseteq M_{\mathbb Q}$
of any ample divisor of $X$, in particular
each facet of $\Delta$ corresponds to
a one dimensional cone of $\Sigma$.
Since $X$ is smooth each maximal cone 
of $X$ is simplicial so that each vertex 
of $\Delta$ has valence $n := \dim(X)$.

When $X$ is a threefold of Picard rank
three the polytope $\Delta$ has six facets
and all its vertexes have valency three.
If we denote by $(v,e,f)$ the vector 
consisting of the number of vertexes, 
edges and facets of $\Delta$, then
the conditions $f = 6$, $2e=3v$
and $v-e+f=2$ give $(v,e,f) = (8,12,6)$.
According to the classification of
hexahedra there are only two possible
topological types for $\Delta$, displayed
in the following pictures.\\

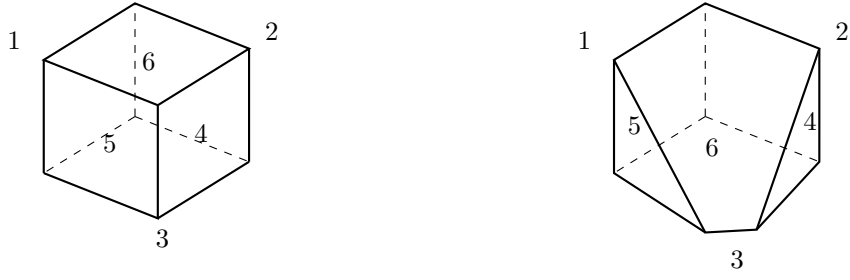
\begin{center}
\begin{tikzpicture}[scale=1.5,x={(1cm,-0.4cm)},y={(-0.8cm,-0.5cm)},z={(0cm,1cm)}]
\coordinate (O) at (0,0,0);
\coordinate (A) at (0,1,0);
\coordinate (B) at (0,1,1);
\coordinate (C) at (0,0,1);
\coordinate (D) at (1,0,0);
\coordinate (E) at (1,1,0);
\coordinate (F) at (1,1,1);
\coordinate (G) at (1,0,1);
\draw[thick] (A) -- (E) -- (F);
\draw[thick] (A) -- (B);
\draw[thick] (E) -- (D);
\draw[thick] (G) -- (D);
\draw[thick] (C) -- (B) -- (F) -- (G) -- cycle;
\draw[dashed] (O) -- (C);
\draw[dashed] (O) -- (A);
\draw[dashed] (O) -- (D);
\node at (-1.3,-0.3) {$\small 1$};
\node at (0,-1.5) {$\small 2$};
\node at (1.2,1.2) {$\small 3$};
\node at (0.1,0.4) {$\small 5$};
\node at (0.5,-0.1) {$\small 4$};
\node at (-0.4,-0.65) {$\small 6$};

\begin{scope}[xshift=5cm]
\coordinate (O) at (0,0,0);
\coordinate (A) at (0,0,1);
\coordinate (B) at (0,1,0);
\coordinate (C) at (1,0,0);
\coordinate (D) at (1,0,1);
\coordinate (E) at (0,1,1);
\coordinate (F) at (1.25,1,0);
\coordinate (G) at (1,1.25,0);
\draw[thick] (A) -- (D) -- (F) -- (G) -- (E) -- cycle;
\draw[thick] (D) -- (C) -- (F);
\draw[thick] (E) -- (B) -- (G);
\draw[dashed] (O) -- (A);
\draw[dashed] (O) -- (C);
\draw[dashed] (O) -- (B);
\node at (-1.3,-0.3) {$\small 1$};
\node at (0,-1.5) {$\small 2$};
\node at (1.4,1.4) {$\small 3$};
\node at (-0.3,0.4) {$\small 5$};
\node at (0.6,-0.4) {$\small 4$};
\node at (0.3,0.3) {$\small 6$};
\end{scope} 
\end{tikzpicture}
\end{center}

Number the facets of each polytope 
from $1$ to $6$ and let $C_i$ be the 
primitive generator of the inward normal
vector to the $i$-th facet. 
Any vertex is labeled by a triple $(i,j,k)$
in such a way that
\begin{equation}
 \label{sm}
 \det(C_i,C_j,C_k) = 1.
\end{equation}
For the type I polytope the vertexes
labels are: $(1,2,3)$, $(2,4,3)$, $(1,6,2)$, 
$(1,3,5)$, $(4,2,6)$, $(1,5,6)$, $(3,4,5)$, 
$(4,6,5)$, while
for the type II polytope the vertexes
labels are: $(1,2,3)$, $(1,3,5)$, $(1,6,2)$, $(2,4,3)$,
$(1,5,6)$, $(4,2,6)$, $(6,5,3)$, $(3,4,6)$.
Without loss of generality we can assume
$\{C_1,C_2,C_3\}$ to be the canonical 
basis of $N_{\mathbb Q}\simeq \mathbb Q^3$.
Thus the matrix whose columns are 
the primitive generators of the fan $\Sigma$
and its orthogonal have the form
\[
 P := 
 \begin{bmatrix}
  1&0&0&a_1&a_2&a_3\\
  0&1&0&b_1&b_2&b_3\\
  0&0&1&c_1&c_2&c_3
 \end{bmatrix}
 \qquad
 \qquad
 Q := 
 \begin{bmatrix}
  -a_1&-b_1&-c_1&1&0&0\\
  -a_2&-b_2&-c_2&0&1&0\\
  -a_3&-b_3&-c_3&0&0&1
 \end{bmatrix}
\]

Applying conditions~\eqref{sm} to all 
type I triples gives the following 
equations 
$a_1 = b_2 = c_3 = -1$,
$ a_3c_1 = b_3c_2 =  a_2b_1 = 0$,
$a_2b_3c_1 + a_3b_1c_2 = 0$.
These equations cut out the union
of six three-dimensional affine spaces
which are in a unique orbit of the
$S_3$ action which permutes the 
coordinates.
One of these spaces is given by
$a_1 = b_2 = c_3 = -1$,
$b_1 = c_1 = c_2 = 0$.
The corresponding $Q$ matrix is
\[
 Q_1 := 
 \begin{bmatrix}
   1&0&0&1&0&0\\
  -a_2&1&0&0&1&0\\
  -a_3&-b_3&1&0&0&1
 \end{bmatrix}
\]

Any integer 
vector $(a_2,a_3,b_3)$ gives an
example of a smooth toric threefold
whose defining fan is of type I.

Applying conditions~\eqref{sm} to all 
the type II triples gives the following 
equations 
$a_1 = b_2 = c_3 = -1$,
$b_3c_2 = a_3c_1 = 0$,
$a_2b_3 + a_3 + 1 = 0$,
$b_3 + a_3b_1 + 1 = 0$.
These equations cut out the union
of three irreducible subvarieties 
of dimension two. Two of these 
are the affine spaces
$a_1 = b_2 = b_3 = c_3 = -1$,
$a_3 = c_2 = 0$, $a_2 = 1$
and 
$a_1 = b_2 = a_3 = c_3 = -1$,
$b_3 = c_1 = 0$, $b_1 = 1$.
The $Q$ matrix for points on the 
first affine space is
\[
 Q_2 := 
 \begin{bmatrix}
   1&-b_1&-c_1&1&0&0\\
   -1&1&0&0&1&0\\
   0&1&1&0&0&1
 \end{bmatrix}
\]

The second affine space gives a similar 
$Q$ matrix.
The third variety has equations
$a_1 = b_2 = c_3 = -1$,
$c_1 = c_2 = 0$,
$a_2b_3 + a_3 + 1 = 0$,
$a_3b_1 + b_3 + 1 = 0$. Considering only integer points this third  variety is the union of a point and two one parameter families corresponding to the following grading matrices
\[
 Q_3 := 
 \begin{bmatrix}
   1&1&0&1&0&0\\
   -1&1&0&0&1&0\\
   0&1&1&0&0&1
 \end{bmatrix}
,
 Q_4 := 
 \begin{bmatrix}
   1&-1&0&1&0&0\\
   -a_2&1&0&0&1&0\\
   1&0&1&0&0&1
 \end{bmatrix}
,
 Q_5 := 
 \begin{bmatrix}
   1&0&0&1&0&0\\
   -a_2&1&0&0&1&0\\
   -a_2+1&1&1&0&0&1
 \end{bmatrix}.
\]

Note now that the case $Q_3$ is a sub-case of $Q_2$ and the case $Q_5$ is a sub-case of $Q_1$. Therefore it is enough to analyze $Q_1,Q_2$ and $Q_4.$ We will call $A,B,\dots,F$ the columns of $Q.$ Note that in all cases $D=(1,0,0),E=(0,1,0)$ and $F=(0,0,1).$

In the case of $Q_1$ we have $C=F$ and the point $B$ moves in the half line from $F$ to $E$. There are three possibilities $b_3<0,b_3=0,b_3>0.$ If $b_3=0$ there is nine possibilities for the point $A$ according to $a_2$ and $a_3$ being negative, zero or positive. If $b_3<0$ we now have four more possibilities depending if the point $A$ is above or below the line $DF,$ therefore $13$ possibilities. For $b_3>0$ we have $13$ possibilities as well. Summing up we have $35$ possibilities. Choosing a particular value of $b_2,a_2,a_3$ for each of the $35$ cases and computing the Mori chamber and the stable base locus decomposition one gets exactly two cases where the Mori chamber and stable base locus decomposition do not coincide, corresponding to the following matrices 
\[
 G_1 := 
 \begin{bmatrix}
   1&0&0&1&0&0\\
   \alpha&1&0&0&1&0\\
   \beta&0&1&0&0&1
 \end{bmatrix}, \alpha,\beta >0, \ 
  G_2 := 
 \begin{bmatrix}
   1&0&0&1&0&0\\
   \alpha&1&0&0&1&0\\
   \beta&\gamma&1&0&0&1
 \end{bmatrix}, \alpha,\beta ,\gamma<0
\] 

Now, consider the case $Q_2$. In this case $A$ is fixed, $B$ moves in the line connecting $(0,1,1)$ to $D$ and $C$ moves in the line $DF.$ If $c_1=0$ then $C=D$ and $B$ has three possibilities according to sign of $b_1$. If $c_1<0$ there are five possibilities for $b_1$ depending on the sign and on which of the inequalities $b_1\geq c_1$ and $b_1 < c_1$ holds.
For $c_1>0$ there are five possibilities as well but now what matters is the sign of $b_1$ and if it is greater than $c_1+1$ or not. Therefore there are $13$ possibilities to check. Computing explicitly the decompositions we get $4$ new possible types of varieties whose Mori chamber and stable base locus decompositions do not coincide:
\[
\begin{array}{ll}
G_3 := 
 \begin{bmatrix}
   1&0&\gamma&1&0&0\\
   -1&1&0&0&1&0\\
   0&1&1&0&0&1
 \end{bmatrix} , \gamma>0,\
 & G_4 := 
 \begin{bmatrix}
   1&\beta&\gamma&1&0&0\\
   -1&1&0&0&1&0\\
   0&1&1&0&0&1
 \end{bmatrix}, \beta>0>\gamma\\[20pt]
  G_5 := 
 \begin{bmatrix}
   1&\beta&\gamma&1&0&0\\
   -1&1&0&0&1&0\\
   0&1&1&0&0&1
 \end{bmatrix} , \beta<0<\gamma, \ 
 & G_6 := 
 \begin{bmatrix}
   1&\beta&\gamma&1&0&0\\
   -1&1&0&0&1&0\\
   0&1&1&0&0&1
 \end{bmatrix} , \beta<\gamma-1<-1
\end{array}
\]

Finally, doing the same with $Q_4$ we have to consider four cases:
$a_2<0$ ($A$ inside the triangle $DEF$),
$a_2=0$ ($A$ in the midpoint of $CF$),
$a_2=1$ ($A$ in the line $BF$), and $a_2>2$ ($A$ below the line $BF$).
Doing this we get a single new type:
\[G_7 := 
 \begin{bmatrix}
   1&-1&0&1&0&0\\
   1&1&0&0&1&0\\
   1&0&1&0&0&1
 \end{bmatrix} 
\]
This proves Theorem~\ref{prop3foldpic3}.

\section{Mori dream spaces of Picard rank two}\label{mainSec}
Let $X$ be a Mori dream space with 
divisor class group ${\rm Cl}(X)$ of 
rank two.
Since $X$ is a projective variety its 
effective cone is pointed. Moreover ${\rm Cl}(X)$ 
has rank two so that we can fix a total order
on the classes in the effective cone $w\leq w'$ 
if $w$ is on the left of $w'$. Given two convex cones 
$\lambda,\lambda'$ contained in the effective cone 
we will write
\[
 \lambda\leq \lambda'
 \quad
 \text{if $w\leq w'$ for any $w\in\lambda$ 
 and $w'\in\lambda$}
\]
Denote by $\{f_1,\dots,f_r\}$ a minimal set 
of homogeneous generators for the Cox ring 
$\mathcal{R}(X)$ of $X$, and let 
$w_i = \deg(f_i)$ for any $i$.

\begin{Proposition}\label{sm_toric}
Let $X$ be a projective $\mathbb{Q}$-factorial toric variety with Picard rank two. Then the Mori chamber and the stable base locus decompositions of $\Eff(X)$ coincide. 
\end{Proposition}
\begin{proof}
Let $\lambda_A$ be the 
semi-ample cone of $X$ and let $\lambda'$,
$\lambda''$ be two distinct maximal GIT 
chambers of $X$. 

According to~\eqref{equality}
it suffices to show that there exists an orbit 
cone of the bunch which contains one of 
$\lambda_A\cup\lambda'$, $\lambda_A\cup\lambda''$
but not the other. Since $X$ is complete
the effective cone is pointed, so that,
since the Picard rank is two, one can
order the GIT chambers of $X$. 

Assuming $\lambda'\leq\lambda''$, we have three possibilities:
either $\lambda'\leq \lambda_A\leq \lambda''$,
or $\lambda_A\leq \lambda'\leq \lambda''$,
or $\lambda'\leq \lambda''\leq \lambda_A$.
Since $X$ is toric each pair of degrees of 
generators of the Cox ring span an orbit
cone. Thus in the first two cases we
can find an orbit cone which contains 
$\lambda_A\cup\lambda'$ but not 
$\lambda_A\cup\lambda''$, while in the 
last case we can find an orbit cone which 
contains $\lambda_A\cup\lambda''$ but not 
$\lambda_A\cup\lambda'$.
\end{proof}

The following is our first simple criterion implying the equality of the two chamber decompositions for a $\mathbb{Q}$-factorial Mori dream space of Picard rank two.

\begin{Proposition}\label{crit1}
Let $X$ be a projective $\mathbb{Q}$-factorial Mori dream space with Picard rank two. If all the generators of $\mathcal{R}(X)$ appear in the walls of the stable base locus decomposition of $\Eff(X)$ then the Mori chamber and the stable base locus decompositions of $\Eff(X)$ coincide. 
\end{Proposition}
\begin{proof}
By (\ref{gitch}) the Mori chamber decomposition is a subdivision of $\Eff(X)$ whose walls are given by some of the generators of $\mathcal{R}(X)$, and it is a refinement of the stable base locus decomposition. Since, by hypothesis all the generators appear as walls of the stable base locus decomposition such a refinement must be trivial.
\end{proof}

Now, we develop some technical results in order to describe the semi-stable loci corresponding to the GIT chambers of the Mori chamber decomposition of a Mori dream space of Picard rank two.

\begin{Lemma}\label{le:ss}
Let $X$ be a Mori dream space with 
Picard rank two, $\lambda\subseteq
{\rm Cl}_{\mathbb Q}(X)$ a maximal 
GIT chamber of $X$, and $\overline 
X^{\rm ss}(\lambda)$ the corresponding
subset of semi-stable points of $\overline X$.
Then the following holds
\[
 \overline X\setminus\overline X^{\rm ss}(\lambda)
 =
 V(f_i\, :\, w_i\leq\lambda)\cup V(f_i\, :\, \lambda\leq w_i)
\]
\end{Lemma}
\begin{proof}
By~\cite[Theorem 3.1.2.8]{ADHL15} we have $\overline X^{\rm ss}(\lambda) =
\overline X^{\rm ss}(w)$ for any $w\in\lambda^\circ$, where the second semi-stable locus
is the complement of the zero set of all the homogeneous 
sections of the Cox ring whose degree 
is a positive multiple of $w$.
If we choose such a class
$w\in\lambda^\circ$ so that $w < w_i$ 
for any $w_i\in\lambda^\circ$ then each
monomial in $f_1,\dots,f_r$ of degree 
$nw$ must contain at least one $f_i$
with $w_i\leq \lambda$. Thus the inclusion
\[
 V(f_i\, :\, w_i\leq\lambda)
 \subseteq 
 \overline X\setminus\overline X^{\rm ss}(\lambda)
\]
follows. The analogous inclusion for $V(f_i\, :\, \lambda\leq w_i)$
can be proved in a similar way. 

To prove the opposite inclusion observe
that if $\overline x\in\overline X$ is a point
which does not belong to the union 
$V(f_i\, :\, w_i\leq\lambda) \cup 
V(f_i\, :\, \lambda\leq w_i)$, then there exist
two sections $f_i$, with $w_i\leq\lambda$, 
and $f_j$, with $\lambda\leq w_j$, each of 
which does not vanish on $\overline x$.

Take non negative $a,b\in \mathbb{Z}$ such that $aw_i+bw_j\in \lambda^\circ$.
Since $f_i^af_j^b$ is a homogeneous element 
of the Cox ring of degree $aw_i+bw_j\in
\lambda^\circ$ which does not 
vanish on $\overline x$, the point $\overline x$
is in $\overline X^{\rm ss}(\lambda)$.
\end{proof}

\begin{Lemma}\label{le:gitloc}
Let $X$ be a Mori dream space with 
Picard rank two and let $\lambda,\lambda'\subseteq
{\rm Cl}_{\mathbb Q}(X)$ be two maximal
distinct GIT chambers of $X$ with $\lambda\leq\lambda'$.
Then the following inclusion is strict
\[
 V(f_i\, :\, w_i\leq\lambda')\subsetneqq V(f_i\, :\, w_i\leq\lambda)
\]
\end{Lemma}
\begin{proof}
Assume that the equality 
$V(f_i\, :\, w_i\leq\lambda') = 
V(f_i\, :\, w_i\leq\lambda)$ holds.
By hypothesis the inclusion
$V(f_i\, :\, \lambda\leq w_i) \subseteq 
V(f_i\, :\, \lambda'\leq w_i)$ holds.
Thus by Lemma~\ref{le:ss} there would 
be an inclusion 
$\overline X^{\rm ss} (\lambda')
 \subseteq \overline X^{\rm ss} (\lambda)$.
By \cite[Theorem 3.1.2.8]{ADHL15} the latter
inclusion would imply $\lambda\subseteq\lambda'$,
a contradiction. 
\end{proof}

Recall that a Mori dream space $X$ is a good quotient of its 
characteristic space $\widehat X = \overline X^{\rm ss}(\lambda_A)$, and denote by $p_{\lambda_A}\colon \widehat X\to X$ the good quotient map. The following simple characterization of stable base loci will be fundamental for the rest of the paper.

\begin{Lemma}\label{SBLc}
If $\lambda\leq\lambda_A$ then the stable base locus of a
class $w\in\lambda$ is 
\stepcounter{thm}
\begin{equation}\label{sbl1}
 {\bf B}(w) 
  =
 p_{\lambda_A}(\widehat X\setminus \overline X^{\rm ss}(\lambda))
  =
 p_{\lambda_A}(\widehat X\cap V(f_i\, :\, w_i\leq \lambda))
\end{equation}
If $\lambda\geq\lambda_A$ then the stable base locus of a
class $w\in\lambda$ is 
\stepcounter{thm}
\begin{equation}\label{sbl2}
 {\bf B}(w) 
  =
 p_{\lambda_A}(\widehat X\setminus \overline X^{\rm ss}(\lambda))
  =
 p_{\lambda_A}(\widehat X\cap V(f_i\, :\, w_i\geq \lambda))
\end{equation}
\end{Lemma}
\begin{proof}
In order to prove (\ref{sbl1}) just note that the first equality holds by definition
while the second equality is due to Lemma~\ref{le:ss} and the fact that $\lambda\leq\lambda_A$. Clearly (\ref{sbl2}) can be proved using a completely analogous argument. 
\end{proof}

The following is the main technical tool of the paper.

\begin{thm}\label{main}
Let $X = X(\lambda_A)$ be a $\mathbb Q$-factorial Mori dream space with 
Picard rank two corresponding to the maximal chamber $\lambda_A$ of the Mori chamber decomposition of $\Eff(X)$. If for any $\lambda'\leq \lambda\leq \lambda_A$ we have
\stepcounter{thm}
\begin{equation}\label{inc1}
V(f_i\, :\, w_i\leq \lambda)\setminus V(f_i\, :\, w_i\geq \lambda_A)\subsetneqq V(f_i\, :\, w_i\leq \lambda')\setminus V(f_i\, :\, w_i\geq \lambda_A)
\end{equation}
and for any $\lambda_A\leq \lambda\leq \lambda'$ we have 
\stepcounter{thm}
\begin{equation}\label{inc2}
V(f_i\, :\, w_i\geq \lambda)\setminus V(f_i\, :\, w_i\leq \lambda_A)\subsetneqq V(f_i\, :\, w_i\geq \lambda')\setminus V(f_i\, :\, w_i\leq \lambda_A)
\end{equation}
then the Mori chamber and the stable base locus decompositions of $\Eff(X)$ coincide.

Furthermore, if for any $\lambda\leq \lambda_A\leq \lambda'$ we have that 
$$V(f_i\, :\, w_i\leq \lambda)\setminus (V(f_i\, :\, w_i\leq \lambda_A)\cup V(f_i\, :\, w_i\geq \lambda_A))$$
is different from
$$V(f_i\, :\, w_i\geq \lambda')\setminus (V(f_i\, :\, w_i\leq \lambda_A)\cup V(f_i\, :\, w_i\geq \lambda_A))$$
then the stable base locus chambers of $\Eff(X)$ are convex.
\end{thm}
\begin{proof}
Let $\lambda$ be a maximal non-ample GIT chamber of $X$. Assume that $\lambda\leq \lambda_A$, where $\lambda_A$ is the ample cone of $X$. By (\ref{sbl1}) in Lemma \ref{SBLc} the stable base locus of a
class $w\in\lambda$ is 
$$ 
{\bf B}(w) 
  =
 p_{\lambda_A}(\widehat X\setminus \overline X^{\rm ss}(\lambda))
  =
 p_{\lambda_A}(\widehat X\cap V(f_i\, :\, w_i\leq \lambda))
$$
Now let $\lambda'\leq\lambda_A$
be any maximal GIT chamber distinct
from $\lambda$ and $\lambda_A$.
Without loss of generality we can assume that $\lambda'\leq\lambda$ then
by Lemma~\ref{le:gitloc} we 
deduce the following
$$
\begin{array}{l}
 V(f_i\, :\, w_i\leq \lambda_A)
 \subsetneqq
 V(f_i\, :\, w_i\leq \lambda)
 \subsetneqq
 V(f_i\, :\, w_i\leq \lambda')\\[2pt]
 V(f_i\, :\, w_i\geq \lambda_A)
 \supsetneqq
 V(f_i\, :\, w_i\geq\lambda)
 \supsetneqq
 V(f_i\, :\, w_i\geq \lambda')
\end{array}
$$
where all the inclusions are strict. Taking the intersection with $\widehat X$ is equivalent to remove from $V(f_i\, :\, w_i\leq \lambda)\subsetneq V(f_i\, :\, w_i\leq \lambda')$ the common subset $V(f_i\, :\, w_i\leq \lambda_A)$ and their intersection with $V(f_i\, :\, w_i\geq \lambda_A)$. So hypothesis (\ref{inc1}) yields that
$$
 \widehat X\cap V(f_i\, :\, w_i\leq \lambda)
 \neq
 \widehat X\cap V(f_i\, :\, w_i\leq \lambda')
$$
Since $X$ is $\mathbb Q$-factorial, the good quotient $p_{\lambda_A}\colon \widehat X\to X$
is geometric~\cite[Corollary 1.6.2.7]{ADHL15}. It follows that the images of the above sets via $p_{\lambda_A}$ remain distinct in $X$ and thus that ${\bf B}(w)\neq {\bf B}(w')$ for any $w\in\lambda^\circ$ and $w'\in\lambda'^\circ$.

Now, assume that $\lambda_A\leq \lambda$. Then (\ref{sbl2}) in Lemma \ref{SBLc} yields that the stable base locus of a class $w\in\lambda$ is 
$$
 {\bf B}(w) 
  =
 p_{\lambda_A}(\widehat X\setminus \overline X^{\rm ss}(\lambda))
  =
 p_{\lambda_A}(\widehat X\cap V(f_i\, :\, w_i\geq \lambda))
$$
In this case if $\lambda'$ is a maximal chamber distinct from $\lambda$ such that $\lambda_A\leq\lambda\leq\lambda'$ Lemma~\ref{le:ss} yields the following strict inclusions
$$
\begin{array}{l}
 V(f_i\, :\, w_i\geq \lambda_A)
 \subsetneqq
 V(f_i\, :\, w_i\geq \lambda)
 \subsetneqq
 V(f_i\, :\, w_i\geq \lambda')\\[2pt]
 V(f_i\, :\, w_i\leq\lambda_A)
 \supsetneqq
 V(f_i\, :\, w_i\leq\lambda)
 \supsetneqq
 V(f_i\, :\, w_i\leq\lambda')
\end{array}
$$
To conclude it is enough to argue as in the previous case using (\ref{inc2}) instead of (\ref{inc1}).

Summing up we showed that any pair of distinct GIT chambers lying on the same side of $\lambda_A$ gives two different stable base locus chambers. Therefore, the Mori chamber decomposition of $\Eff(X)$ coincide with its stable base locus decomposition. 

Finally, an analogous argument shows that if $\lambda\leq \lambda_A\leq\lambda'$ and our last hypothesis holds then for any $w\in \lambda$ and $w'\in\lambda'$ we have $\textbf{B}(w)\neq \textbf{B}(w)$, and hence the stable base locus chambers of $\Eff(X)$ are convex.
\end{proof}

\begin{Remark}
Let us consider the Mori dream space $X$ in Example \ref{ex1}. Note that (\ref{mon}) yields
$$
\begin{array}{l}
V(f_i\, :\, w_i\geq\lambda_A) = \{T_1=T_2=\widetilde{F}=\widetilde{G}=0\}\\[2pt]\end{array}$$
where $\widetilde{F},\widetilde{G}$ are general linear combinations of the following monomials
\stepcounter{thm}
\begin{equation*}
\left\lbrace\begin{array}{llllllll}
T_3T_8 & T_3T_9 & T_3T_{10} & T_3T_{11} & T_4T_8 & T_4T_9 & T_4T_{10} & T_4T_{11};\\ 
T_5T_8 & T_5T_9 & T_5T_{10} & T_5T_{11} & T_6T_8 & T_6T_9 & T_6T_{10} & T_6T_{11};
\end{array}\right. 
\end{equation*} 
$$\begin{array}{l}
V(f_i\, :\, w_i\leq\lambda) = \{T_7=T_8=T_9=T_{10}=T_{11}=0\}\\[2pt]
V(f_i\, :\, w_i\leq\lambda') = \{T_1=T_2=T_8=T_9=T_{10}=T_{11}=0\}\cup\{T_7=T_8=T_9=T_{10}=T_{11}=0\}
\end{array}
$$
and $V(f_i\, :\, w_i\geq\lambda_A)\supset \{T_1=T_2=T_8=T_9=T_{10}=T_{11}=0\}$ contains a component of the set $V(f_i\, :\, w_i\leq\lambda')$.
\end{Remark}

In what follows we work out some interesting consequences of Theorem \ref{main}.

\begin{Corollary}\label{irr}
Let $X = X(\lambda_A)$ be a $\mathbb Q$-factorial Mori dream space with 
Picard rank two corresponding to the maximal chamber $\lambda_A$ of the Mori chamber decomposition of $\Eff(X)$. If for any maximal chamber $\lambda$ we have that $V(f_i\, :\, w_i\leq\lambda)$ and $V(f_i\, :\, w_i\geq\lambda)$ are irreducible then the Mori chamber and the stable base locus decompositions of $\Eff(X)$ coincide.
\end{Corollary}
\begin{proof}
Without loss of generality we may assume that $\lambda'\leq \lambda\leq \lambda_A$. Since $V(f_i\, :\, w_i\leq\lambda')$ is irreducible either $V(f_i\, :\, w_i\geq\lambda_A)\supseteq V(f_i\, :\, w_i\leq\lambda')$ or $V(f_i\, :\, w_i\geq\lambda_A)\cap V(f_i\, :\, w_i\leq\lambda')$ is a closed subset of $V(f_i\, :\, w_i\leq\lambda')$.

Assume that $V(f_i\, :\, w_i\geq\lambda_A)\supseteq V(f_i\, :\, w_i\leq\lambda')$. Then since $V(f_i\, :\, w_i\geq\lambda')\subseteq V(f_i\, :\, w_i\geq\lambda_A)$ we get that $X^{\rm ss}(\lambda_A)\subseteq X^{\rm ss}(\lambda')$, and \cite[Theorem 3.1.2.8]{ADHL15} yields that $\lambda'\subseteq\lambda_A$, a contradiction. Therefore, $V(f_i\, :\, w_i\geq\lambda_A)$ intersects $V(f_i\, :\, w_i\leq\lambda')$ in a closed subset, and to conclude it is enough to apply Theorem \ref{main}.
\end{proof}

Now, we are ready to prove the main result of the paper.

\begin{thm}\label{main2}
Let $X$ be a $\mathbb{Q}$-factorial Mori dream space with Picard rank two, $\{f_1,\dots,f_r\}$ a minimal set of homogeneous generators for the Cox ring $\mathcal{R}(X)$, $w_i := \deg(f_i)$, and $\lambda_A$ be the ample chamber of $X$. Denote by $c$ the codimension of $X$ into its canonical toric embedding~\cite[Section 3.2.5]{ADHL15}. Define 
$$h^+:=\#\{f_i\, :\, w_i\geq\lambda_A\}\quad and \quad h^-:=\#\{f_i\, :\, w_i\leq\lambda_A\}$$ 
If $h^->c$ and $h^+>c$, then the Mori chamber and the stable base locus decomposition of $\Eff(X)$ coincide.
\end{thm}
\begin{proof}
Consider $p_{\lambda_A}:\widehat{X}\rightarrow X$, let $\overline{X}$ be the total space of $X$, and $\overline{Z} \cong \mathbb{A}^{r}$ be the affine space with coordinates given by the $f_i$. Let $\lambda',\lambda$ be two chambers of the Mori chamber decomposition of $\Eff(X)$ as in the following picture
$$
 
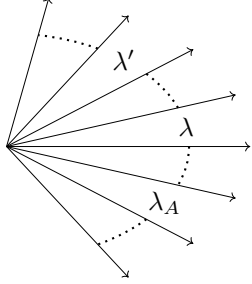
\begin{tikzpicture}[xscale=0.8,yscale=0.5]
 \draw[line width=0mm,->] (0,0) -- (80:4);
 \draw[line width=0mm,->] (0,0) -- (60:4);
 \draw[line width=0mm,->] (0,0) -- (40:4);
 \draw[line width=0mm,->] (0,0) -- (20:4);
 \draw[line width=0mm,->] (0,0) -- (0:4);
 \draw[line width=0mm,->] (0,0) -- (-20:4); 
 \draw[line width=0mm,->] (0,0) -- (-40:4);
 \draw[line width=0mm,->] (0,0) -- (-60:4);
 \node at (50:3) {$\lambda'$};
 \node at (10:3) {$\lambda$};
 \node at (-30:3) {$\lambda_A$};
 \draw[thick,dotted] (-40:3) arc (-40:-60:3);
 \draw[thick,dotted] (0:3) arc (0:-20:3); 
 \draw[thick,dotted] (20:3) arc (20:40:3);  
 \draw[thick,dotted] (60:3) arc (60:80:3);
 \end{tikzpicture}
$$
Recall that by (\ref{sbl1}) in Lemma \ref{SBLc} the stable base loci of classes $w\in\lambda$, $w'\in\lambda'$ are given respectively by 
$$
\begin{array}{l}
\textbf{B}(w)= p_{\lambda_A}(\widehat X\setminus \overline X^{\rm ss}(\lambda)) = p_{\lambda_A}(\widehat X\cap V(f_i\, :\, w_i\leq \lambda))\\ 
\textbf{B}(w')= p_{\lambda_A}(\widehat X\setminus \overline X^{\rm ss}(\lambda')) = p_{\lambda_A}(\widehat X\cap V(f_i\, :\, w_i\leq \lambda'))
\end{array} 
$$
and the non semi-stable locus of $\lambda_A$ is
$$V(f_i\, :\, w_i\leq \lambda_A)\cup V(f_i\, :\, w_i\geq \lambda_A)$$
Now, $\overline{X}\subset\mathbb{A}^{r}$ has dimension $\dim(X)+2$, and hence any irreducible component of the intersection $\overline{X}\cap V(f_i\, :\, w_i\leq \lambda')$ has dimension greater than or equal to $\dim(X)+2-h'$, where $h' = \#\{f_i\, :\, w_i\leq \lambda'\}$. Assume that an irreducible component of $\overline{X}\cap V(f_i\, :\, w_i\leq \lambda')$ is contained in $\overline{X}\cap V(f_i\, :\, w_i\geq \lambda_A)$. Then such component must be contained in 
$$V(f_i,f_j\, :\, w_i\leq \lambda',w_j\geq\lambda_A)$$ 
which has dimension $r-h'-h^{+}$. This forces $h^{+}\leq c$, a contradiction with our hypothesis. Now, to conclude that $\lambda,\lambda'$ are two different stable base locus chambers it is enough to recall that Lemma \ref{le:gitloc} yields $V(f_i\, :\, w_i\leq \lambda) \varsubsetneqq V(f_i\, :\, w_i\leq \lambda')$.

When $\lambda_A\leq \lambda\leq\lambda'$ we argue in a completely analogous way, and then to conclude it is enough to apply Theorem \ref{main}. 
\end{proof}

The following is the first immediate consequence of Theorem \ref{main2}.

\begin{Corollary}\label{hyp}
Let $Z$ be a projective normal $\mathbb{Q}$-factorial toric variety with $\rk(\Cl(Z))=2$, and $X\subseteq Z$ a projective normal $\mathbb{Q}$-factorial Mori dream hypersurface such that $\imath^*\colon {\Cl}(Z)\to{\Cl}(X)$ is an isomorphism. Then the Mori chamber and the stable base locus decompositions of both $\Eff(Z)$ and $\Eff(X)$ coincide.   
\end{Corollary}
\begin{proof}
For a toric variety the claim follows from Proposition \ref{sm_toric} and it is also an immediate consequence of Theorem \ref{main2} with $c=0$. In general, following the notation in the proof of Theorem \ref{main2}, there are always at least two generators in the sets $\{f_i\, :\, w_i\geq\lambda_A\}$, $\{f_i\, :\, w_i\leq\lambda_A\}$ otherwise $\lambda_A$ would be a chamber of $\Eff(X)\setminus \Mov(X)$. Since $c = \codim_{Z}(X) = 1$ we conclude by Theorem \ref{main2}.  
\end{proof}

\begin{Remark}\label{sharp}
Theorem \ref{main2} is sharp. Indeed, the Mori dream space in Example \ref{ex1} has three Mori chamber but just two stable base locus chambers. In this example $h^{+}=c = 2$.
\end{Remark}

\begin{Remark}
An intrinsic quadric is a normal $\mathbb{Q}$-factorial projective Mori dream space with Cox ring defined by a single quadratic relation. Smooth intrinsic quadrics with small Picard rank have been studied recently in \cite{FH18}. By Corollary \ref{hyp} the Mori chamber and the stable base locus decomposition of the effective cone of an intrinsic quadric of Picard rank two coincide. 
\end{Remark}

\subsection{Picard rank two varieties with a torus action of complexity one}\label{comp1}
Recall that a variety with a torus action of complexity one is a normal complete algebraic variety $X$ with an effective action of a torus $T$ such that the biggest $T$-orbits are of codimension one in $X$.

\begin{Proposition}\label{pro:comp1}
Let $X$ be smooth rational projective variety of Picard rank two that admits a torus action of complexity one. Then the Mori chamber and the stable base locus decomposition of $\Eff(X)$ coincide.
\end{Proposition}
\begin{proof}
By \cite[Theorem 1.1]{FHN16} any smooth rational projective variety of Picard rank two with a torus action of complexity one, with just one exception, is a Mori dream hypersurface in its canonical toric embedding. Therefore, with the exception of the variety No. 13 in the statement of \cite[Theorem 1.1]{FHN16} the claim follows directly from Corollary \ref{hyp}.

On the hand, the Cox ring of the exceptional variety $X$ has eight generators $T_1,\dots, T_8$, with $\deg(T_1) = \deg(T_3) = \deg(T_5) = \deg(T_7)$, and $\deg(T_2) = \deg(T_4) = \deg(T_6) = \deg(T_8)$. Therefore, both the Mori chamber and the stable base locus decomposition of $\Eff(X)$ consist of a single chamber which is indeed the nef cone of $X$.  
\end{proof}

In what follows we apply the techniques developed in this section to compute the stable base locus decomposition which by Proposition \ref{pro:comp1} coincide with the Mori chamber decomposition, of the effective cones of the varieties in \cite[Theorem 1.1]{FHN16}.

\begin{Example}(No. 6 in \cite[Theorem 1.1]{FHN16})
In this case $X$ is a variety of dimension $m+3$ with Cox ring given by
$$\mathcal{R}(X)\cong\frac{k[T_1,\dots,T_6,S_1,\dots,S_m]}{(T_1T_2+T_3T_4+T_5^2T_6)}$$
with $m\geq 1$, and grading matrix
$$
\left(\begin{array}{ccccccccc}
0 & 2c+1 & a & b & c & 1 & 1 & \dots & 1\\ 
1 & 1 & 1 & 1 & 1 & 0 & 0 & \dots & 0
\end{array}\right) 
$$
with $a,b,c\geq 0$, $a< b$ and $a+b = 2c+1$. Here we develop the case $0<a<c$, when $a = 0$ or $a=c$ a similar argument will work. Therefore, $\MCD(X)$ is a possibly trivial coarsening of the following decomposition
$$
 
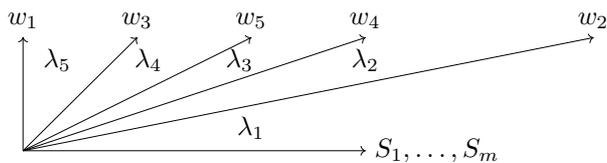
\begin{tikzpicture}[xscale=1.5,yscale=1.5]
 \draw[line width=0mm,->] (0,0) -- (0,1);
 \draw[line width=0mm,->] (0,0) -- (1,1);
 \draw[line width=0mm,->] (0,0) -- (2,1);
 \draw[line width=0mm,->] (0,0) -- (3,1);
 \draw[line width=0mm,->] (0,0) -- (5,1);
 \draw[line width=0mm,->] (0,0) -- (3,0);
\node[above] at (0,1) {$w_1$};
\node[above] at (1,1) {$w_3$};
\node[above] at (2,1) {$w_5$};
\node[above] at (3,1) {$w_4$};
\node[above] at (5,1) {$w_2$};
\node[right] at (3,0) {$S_1,\dots,S_m$};
\node at (2,0.2) {$\lambda_1$};
\node at (3,0.8) {$\lambda_2$};
\node at (1.9,0.8) {$\lambda_3$};
\node at (1.1,0.8) {$\lambda_4$};
\node at (0.3,0.8) {$\lambda_5$};
 \end{tikzpicture}
$$
where $\lambda_1 = \lambda_A$ is the ample cone of $X$. Note that (\ref{sbl1}) in Lemma \ref{SBLc} yields
$$
\begin{normalsize}
\begin{array}{ll}
{\bf B}(w) = p_{\lambda_A}(\widehat X\cap V(f_i\, :\, w_i\leq \lambda_2)) = p_{\lambda_A}(\widehat X\cap \{T_4 = T_5 = T_3 = T_1 = 0\}) & \text{if} \: w\in \lambda_2;\\ 
{\bf B}(w) = p_{\lambda_A}(\widehat X\cap V(f_i\, :\, w_i\leq \lambda_3)) = p_{\lambda_A}(\widehat X\cap \{T_5 = T_3 = T_1 = 0\}) & \text{if}\: w\in \lambda_3;\\ 
{\bf B}(w) = p_{\lambda_A}(\widehat X\cap V(f_i\, :\, w_i\leq \lambda_4)) = p_{\lambda_A}(\widehat X\cap \{T_3 = T_1 = T_5^2T_6 = 0\}) & \text{if}\: w\in \lambda_4;\\
{\bf B}(w) = p_{\lambda_A}(\widehat X\cap V(f_i\, :\, w_i\leq \lambda_5)) = p_{\lambda_A}(\widehat X\cap \{T_1 = T_3T_4+T_5^2T_6 = 0\}) & \text{if}\: w\in \lambda_5.
\end{array} 
\end{normalsize}
$$
Therefore, $\MCD(X) = \SBLD(X) = \{\lambda_A,\lambda_2,\lambda_3,\lambda_4,\lambda_5\}$. 
\end{Example}

\begin{Example}(No. 8 in \cite[Theorem 1.1]{FHN16}) 
In this case $X$ is a variety of dimension $m+3$ with Cox ring given by
$$\mathcal{R}(X)\cong\frac{k[T_1,\dots,T_6,S_1,\dots,S_m]}{(T_1T_2+T_3T_4+T_5T_6)}$$
with $m\geq 2$, and grading matrix
$$
\left(\begin{array}{cccccccccc}
0 & 0 & 0 & 0 & 0 & 0 & 1 & 1   &\dots & 1\\ 
1 & 1 & 1 & 1 & 1 & 1 & 0 & a_2 & \dots & a_m
\end{array}\right) 
$$
with $0\leq a_2\leq \dots \leq a_m$ and $a_m > 0$. We develop the case $0< a_2 < \dots < a_m$, the same argument will work in the remaining cases as well. Therefore, $\MCD(X)$ is a possibly trivial coarsening of the following decomposition
$$
 
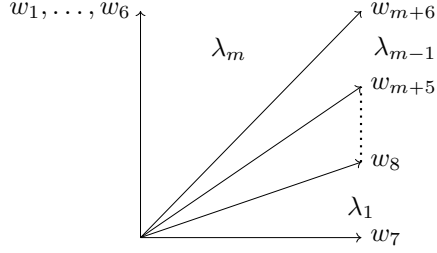
\begin{tikzpicture}[xscale=2.9,yscale=1.0]
 \draw[line width=0mm,->] (0,0) -- (1,0);
 \draw[line width=0mm,->] (0,0) -- (1,1);
 \draw[thick,dotted] (1,1) -- (1,2); 
 \draw[line width=0mm,->] (0,0) -- (1,2);
 \draw[line width=0mm,->] (0,0) -- (1,3);
 \draw[line width=0mm,->] (0,0) -- (0,3);
\node[right] at (1,0) {$w_7$};
\node[right] at (1,1) {$w_8$};
\node[right] at (1,2) {$w_{m+5}$};
\node[right] at (1,3) {$w_{m+6}$};
\node[left] at (0,3) {$w_1,\dots,w_6$};
\node at (1,0.4) {$\lambda_1$};
\node at (1.2,2.5) {$\lambda_{m-1}$};
\node at (0.4,2.5) {$\lambda_m$};
 \end{tikzpicture}
$$
where $\lambda_m = \lambda_A$ is the ample cone of $X$. Note that (\ref{sbl2}) in Lemma \ref{SBLc} yields
$$
\begin{normalsize}
\begin{array}{l}
{\bf B}(w) = p_{\lambda_A}(\widehat X\cap V(f_i\, :\, w_i\geq \lambda_j)) = p_{\lambda_A}(\widehat X\cap \{S_j = \dots = S_1 = T_1T_2+T_3T_4+T_5T_6 = 0\}) 
\end{array} 
\end{normalsize}
$$
if $w\in \lambda_j$, for $j = 1,\dots,m-1$. Therefore, $\MCD(X) = \SBLD(X) = \{\lambda_A,\lambda_{m-1},\dots,\lambda_1\}$.
\end{Example}

For all the other varieties listed in \cite[Theorem 1.1]{FHN16}, with the exception of the varieties No. 3 and No. 12 for generic parameters, arguing similarly we get that the Mori chamber decomposition of the variety coincide with the one of the ambient toric variety which is given by the corresponding grading matrix in \cite[Theorem 1.1]{FHN16}. In the following we study the two exceptional cases.   

\begin{Example}(No. 3 in \cite[Theorem 1.1]{FHN16})\label{no3}
In this case $X$ is a $3$-fold with Cox ring given by
$$\mathcal{R}(X)\cong\frac{k[T_1,\dots,T_6]}{(T_1T_2T_3^2+T_4T_5+T_6^2)}$$
with $m\geq 2$, and grading matrix
$$
\left(\begin{array}{cccccc}
0 & 0 & 1 & 1 &   1 & 1\\ 
1 & 1 & 0 & 2-a & a & 1
\end{array}\right) 
$$
with $a\geq 1$. Therefore, in the case $a\geq 3$ the $\MCD(X)$ is a possibly trivial coarsening of the following decomposition
$$
 
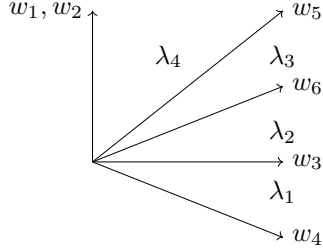
\begin{tikzpicture}[xscale=2.5,yscale=1.0]
 \draw[line width=0mm,->] (0,0) -- (1,0);
 \draw[line width=0mm,->] (0,0) -- (1,1); 
 \draw[line width=0mm,->] (0,0) -- (1,2);
 \draw[line width=0mm,->] (0,0) -- (1,-1);
 \draw[line width=0mm,->] (0,0) -- (0,2);
\node[right] at (1,0) {$w_3$};
\node[right] at (1,1) {$w_6$};
\node[right] at (1,2) {$w_5$};
\node[right] at (1,-1) {$w_4$};
\node[left] at (0,2) {$w_1,w_2$};
\node at (1,-0.4) {$\lambda_1$};
\node at (1,0.4) {$\lambda_2$};
\node at (1,1.4) {$\lambda_3$};
\node at (0.4,1.4) {$\lambda_4$}; \end{tikzpicture}
$$
where $\lambda_4 = \lambda_A$ is the ample cone of $X$. Note that (\ref{sbl2}) in Lemma \ref{SBLc} yields
$$
\begin{normalsize}
\begin{array}{ll}
{\bf B}(w) = p_{\lambda_A}(\widehat X\cap V(f_i\, :\, w_i\geq \lambda_3)) = p_{\lambda_A}(\widehat X\cap \{T_6 = T_3 = T_4 = 0\}) & \text{if} \: w\in \lambda_3;\\ 
{\bf B}(w) = p_{\lambda_A}(\widehat X\cap V(f_i\, :\, w_i\geq \lambda_2)) = p_{\lambda_A}(\widehat X\cap \{T_6 = T_3 = T_4 = 0\}) & \text{if}\: w\in \lambda_2;\\ 
{\bf B}(w) = p_{\lambda_A}(\widehat X\cap V(f_i\, :\, w_i\geq \lambda_1)) = p_{\lambda_A}(\widehat X\cap \{T_4 = 0\}) & \text{if}\: w\in \lambda_1.
\end{array} 
\end{normalsize}
$$
Therefore, $\MCD(X) = \SBLD(X) = \{\lambda_A,\lambda_2\cup\lambda_3,\lambda_1\}$.

In the case $a=2$, $\MCD(X)$ is a possibly trivial coarsening of the following decomposition
$$
 
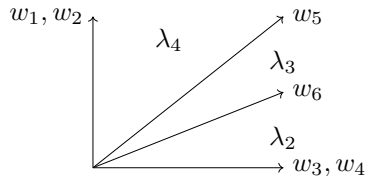
\begin{tikzpicture}[xscale=2.5,yscale=1.0]
 \draw[line width=0mm,->] (0,0) -- (1,0);
 \draw[line width=0mm,->] (0,0) -- (1,1); 
 \draw[line width=0mm,->] (0,0) -- (0,2);
 \draw[line width=0mm,->] (0,0) -- (1,2);
\node[right] at (1,0) {$w_3,w_4$};
\node[right] at (1,1) {$w_6$};
\node[right] at (1,2) {$w_5$};
\node[left] at (0,2) {$w_1,w_2$};
\node at (1,0.4) {$\lambda_2$};
\node at (1,1.4) {$\lambda_3$};
\node at (0.4,1.7) {$\lambda_4$}; \end{tikzpicture}
$$
where $\lambda_4 = \lambda_A$ is the ample cone of $X$. Note that (\ref{sbl2}) in Lemma \ref{SBLc} yields
$$
\begin{normalsize}
\begin{array}{ll}
{\bf B}(w) = p_{\lambda_A}(\widehat X\cap V(f_i\, :\, w_i\geq \lambda_2)) = p_{\lambda_A}(\widehat X\cap \{T_6 = T_3= T_4= 0\}) & \text{if} \: w\in \lambda_2;\\
{\bf B}(w) = p_{\lambda_A}(\widehat X\cap V(f_i\, :\, w_i\geq \lambda_3)) = p_{\lambda_A}(\widehat X\cap \{T_6 = T_3= T_4= 0\}) & \text{if} \: w\in \lambda_3.
\end{array} 
\end{normalsize}
$$
Therefore, $\MCD(X) = \SBLD(X) = \{\lambda_A,\lambda_2\cup \lambda_3\}$.

In the case $a=1$, $\MCD(X)$ is a possibly trivial coarsening of the following decomposition
$$
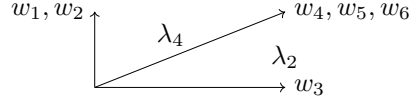
\begin{tikzpicture}[xscale=2.5,yscale=1.0]
\draw[line width=0mm,->] (0,0) -- (1,0);
\draw[line width=0mm,->] (0,0) -- (1,1); 
\draw[line width=0mm,->] (0,0) -- (0,1);
\node[right] at (1,0) {$w_3$};
\node[right] at (1,1) {$w_4,w_5,w_6$};
\node[left] at (0,1) {$w_1,w_2$};
\node at (1,0.4) {$\lambda_2$};
\node at (0.4,0.7) {$\lambda_4$}; \end{tikzpicture}
$$
where $\lambda_4 = \lambda_A$ is the ample cone of $X$. Now (\ref{sbl2}) in Lemma \ref{SBLc} yields
$$
\begin{normalsize}
\begin{array}{ll}
{\bf B}(w) = p_{\lambda_A}(\widehat X\cap V(f_i\, :\, w_i\geq \lambda_2)) = p_{\lambda_A}(\widehat X\cap \{T_3 = 0\}) & \text{if} \: w\in \lambda_2.
\end{array} 
\end{normalsize}
$$
Therefore, $\MCD(X) = \SBLD(X) = \{\lambda_A,\lambda_2\}$.
\end{Example}

\begin{Example}(No. 12 in \cite[Theorem 1.1]{FHN16})
In this case $X$ is a variety of dimension $m+2$ with Cox ring given by
$$\mathcal{R}(X)\cong\frac{k[T_1,\dots,T_5,S_1,\dots,S_m]}{(T_1T_2+T_3T_4+T_5^2)}$$
with $m\geq 2$, and grading matrix
$$
\left(\begin{array}{cccccccc}
1 & 1 & 1 & 1 & 1 & 0 & \dots & 0  \\ 
0 & 2c & a & b & c & 1 & \dots & 1
\end{array}\right) 
$$
with $0\leq a\leq c \leq b$ and $a+b = 2c$. Therefore, in the case $0<a<c<b$, $\MCD(X)$ is a possibly trivial coarsening of the following decomposition
$$
 
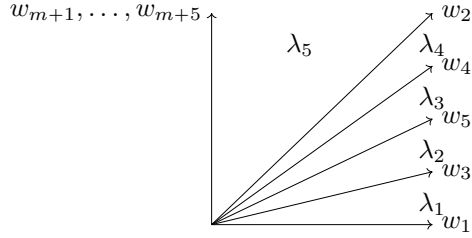
\begin{tikzpicture}[xscale=2.9,yscale=0.7]
 \draw[line width=0mm,->] (0,0) -- (1,0);
 \draw[line width=0mm,->] (0,0) -- (1,1);
 \draw[line width=0mm,->] (0,0) -- (1,2);
 \draw[line width=0mm,->] (0,0) -- (1,3);
 \draw[line width=0mm,->] (0,0) -- (1,4);
 \draw[line width=0mm,->] (0,0) -- (0,4);
\node[right] at (1,0) {$w_1$};
\node[right] at (1,1) {$w_3$};
\node[right] at (1,2) {$w_5$};
\node[right] at (1,3) {$w_4$};
\node[right] at (1,4) {$w_2$};
\node[left] at (0,4) {$w_{m+1},\dots,w_{m+5}$};
\node at (1,0.4) {$\lambda_1$};
\node at (1,1.4) {$\lambda_2$};
\node at (1,2.4) {$\lambda_3$};
\node at (1,3.4) {$\lambda_4$};
\node at (0.4,3.4) {$\lambda_5$};
 \end{tikzpicture}
$$
where $\lambda_5 = \lambda_A$ is the ample cone of $X$. Note that (\ref{sbl2}) in Lemma \ref{SBLc} yields
$$
\begin{normalsize}
\begin{array}{ll}
{\bf B}(w) = p_{\lambda_A}(\widehat X\cap V(f_i\, :\, w_i\geq \lambda_4)) = p_{\lambda_A}(\widehat X\cap \{T_4 = T_5 = T_3 = T_1 = 0\}) & \text{if} \: w\in \lambda_4;\\ 
{\bf B}(w) = p_{\lambda_A}(\widehat X\cap V(f_i\, :\, w_i\geq \lambda_3)) = p_{\lambda_A}(\widehat X\cap \{T_5 = T_3 = T_1 = 0\}) & \text{if}\: w\in \lambda_3;\\ 
{\bf B}(w) = p_{\lambda_A}(\widehat X\cap V(f_i\, :\, w_i\geq \lambda_2)) = p_{\lambda_A}(\widehat X\cap \{T_3 = T_1 = T_5 = 0\}) & \text{if}\: w\in \lambda_2.\\
{\bf B}(w) = p_{\lambda_A}(\widehat X\cap V(f_i\, :\, w_i\geq \lambda_1)) = p_{\lambda_A}(\widehat X\cap \{T_1 = T_3T_4+T_5^2 = 0\}) & \text{if}\: w\in \lambda_1.
\end{array} 
\end{normalsize}
$$
Therefore, $\MCD(X) = \SBLD(X) = \{\lambda_A,\lambda_4,\lambda_2\cup\lambda_3,\lambda_1\}$.

If there is an equality in any of the inequalities $0\leq a \leq c \leq b$, then some $w_j$ coincide and the corresponding chambers collapse as in Example \ref{no3}. For instance, if $a=0$ then $w_1=w_3$ and the chamber $\lambda_1$ does not exist.
\end{Example}

\section{Grassmannians blow-ups}\label{grassbu}
Let $\mathbb{G}(r,n)$ be the Grassmannian parametrizing $r$-planes in $\mathbb{P}^n$, and $\mathbb{G}(r,n)_k$ the blow-up of $\mathbb{P}^n$ at $k$ general points. These blow-ups have been studied in \cite{MR18}. In particular the stable base locus decomposition of $\Eff(\mathbb{G}(r,n)_1)$ has been computed in \cite[Theorem 1.3]{MR18}. 

In this section we will compute the Cox ring of $\mathbb{G}(r,n)_1$ by exploiting its spherical nature, and as a consequence of Proposition \ref{crit1} we will answer positively to \cite[Question 6.9]{MR18} which ask whether the decomposition given in \cite[Theorem 1.3]{MR18} is the Mori chamber decomposition of $\Eff(\mathbb{G}(r,n)_1)$.

\begin{Definition}
A \textit{spherical variety} is a normal variety $X$ together with an action of a connected reductive affine algebraic group $\mathscr{G}$, a Borel subgroup $\mathscr{B}\subseteq \mathscr{G}$, and a base point $x_0\in X$ such that the $\mathscr{B}$-orbit of $x_0$ in $X$ is a dense open subset of $X$. 

Let $(X,\mathscr{G},\mathscr{B},x_0)$ be a spherical variety. We distinguish two types of $\mathscr{B}$-invariant prime divisors: a \textit{boundary divisor} of $X$ is a $\mathscr{G}$-invariant prime divisor on $X$, a \textit{color} of $X$ is a $\mathscr{B}$-invariant prime divisor that is not $\mathscr{G}$-invariant.
\end{Definition}

For instance, any toric variety is a spherical variety with $\mathscr{B}=\mathscr{G}$ equal to the torus. For a toric variety there are no colors, and the boundary divisors are the usual toric invariant divisors.

Set $\Lambda:=\left\{I\subset \{0,\dots,n\}, |I|=r+1 \right\}$ and $N:=|\Lambda|+1$. Define the \textit{Hamming distance} on $\Lambda$ as 
$$d(I,J)=|I|-|I\cap J|=|J|-|I\cap J|$$ 
for each $I,J\in\Lambda$. Note that, with respect to this distance, the diameter of $\Lambda$ is $r+1$.
We consider the Grassmannian $\mathbb{G}(r,n)$ in the usual Pl\"ucker embedding $\mathbb{G}(r,n)\subset \mathbb{P}^{N}$. 

For each pair $I=\{i_0<\dots < i_{r-1}\}, J=\{j_0<\dots <j_{r+1}\}\subset \{0,\dots,n\}$ with $|I|=r,|J|=r+2$, define a quadratic polynomial
\stepcounter{thm}
\begin{equation}\label{pluckereq}
F_{IJ}=\sum_{t=0}^{r+1}(-1)^t
p_{i_0\dots i_{r-1}j_t}
p_{j_1\dots \widehat{j_t}\dots j_{r+1}}
\end{equation}
Then the ideal of $K[p_{I}, I\in \Lambda]$ generated by the $F_{IJ}$ is the ideal defining $\mathbb{G}(r,n)\subset \mathbb{P}^{N}$ \cite[Section I.4]{Sh13}. We denote by $\mathbb{G}(r,n)_1$ the blow-up of $\mathbb{G}(r,n)$ at $p=\left\langle e_0,\dots, e_r\right\rangle$, where $\{e_0,\dots,e_n\}$ is the canonical basis of $K^{n+1}$, by $H$ the pull-back to $\mathbb{G}(r,n)_1$ of the hyperplane section of $\mathbb{P}^N$, and by $E$ the exceptional divisor of the blow-up. 

\begin{Proposition}\label{CoxG1}
In the polynomial ring $K[S,T_I,I\in \Lambda]$ consider the ideal $\mathfrak{J}$ generated by the Pl\"ucker relations \eqref{pluckereq} in the coordinates $T_I$. Then 
$$\mathcal{R}(\G(r,n)_1)\cong \frac{K[S,T_I,I\in \Lambda]}{\mathfrak{J}}$$ 
and the degree of the variable $T_I$ in $\Cl(\G(r,n)_1)=\mathbb{Z}[H]+\mathbb{Z}[E]$ is $(1,-d(I,\{0,\dots,r\}))$, where $\deg(S) = (0,1)$.
\end{Proposition}
\begin{proof}
By \cite[Proposition 4.1]{MR18} under the action of the reductive group
$$\mathscr{G}\!=\!\left\{
\begin{pmatrix}
	A & 0 \\ 0 & B
\end{pmatrix}, A\in GL_{r+1}, B\in GL_{n-r}\: |\: \det(A)\det(B)=1 \right\}\subset SL_{n+1}$$
the blow-up $\G(r,n)_1$ is a spherical variety. We consider the Borel subgroup $\mathscr{B}\subset \mathscr{G}
$ of matrices with upper triangular blocks. Consider the divisors $D_0, \dots, \ D_{r+1}$ in $\G(r,n)$ defined as
$D_j:=\{[\Sigma]\in \G(r,n): \Sigma \cap \Gamma_j \neq \varnothing \}$, where
$$\begin{cases}
\Gamma_0=\left\langle e_{r+1},\dots, e_n\right\rangle;\\
\Gamma_1=\left\langle e_0,e_{r+1},\dots, e_{n-1}\right\rangle;\\
\vdots\\
\Gamma_{r+1}=\left\langle e_0,\dots,e_r,e_{r+1},\dots, e_{n-r-1}\right\rangle;\\
\Gamma'=\left\langle e_0,\dots,e_r\right\rangle. 
\end{cases}$$
Pulling-back these divisors via the blow-up map we obtain divisors in
$\mathbb{G}(r,n)_1$. For sake of simplicity we will use the same notation for divisors in $\mathbb{G}(r,n)$ and their pull-backs in $\mathbb{G}(r,n)_1$.

Now, note that $\mathscr{G}$ preserves the dimension of the intersection of a given subspace of $\mathbb{P}^n$ with $\Gamma_0$ and with $\Gamma'$. Therefore $\mathscr{G}\cdot D_0=D_0$ and $\mathscr{G}\cdot E=E$ that is, $D_0$ and $E$ are boundary divisors. Note also that each $D_j$ is a $\mathscr{B}$-invariant but not a $\mathscr{G}$-invariant prime divisor, and therefore $D_1,\dots, D_{r+1}$ are colors. 

In order to determine the $\mathscr{G}$-orbit of $D_1,\dots,D_{r+1}$ we have to describe these divisors explicitly as zeros of polynomials in the Pl\"ucker coordinates.

In $\mathbb{G}(r,n)$ the divisor $D_0$ is given by $D_0=V(p_{0,\dots, r})$. Indeed, if $q\in D_0$ then $q=[\Sigma]$ with 
$\Sigma \cap \Gamma_0 \neq \varnothing$
and therefore it can be represented with a matrix whose first row is of the form $(0, \dots, 0,a_{r+1},\dots,a_n)$.

This implies that $p_{0\dots r}(q)=0$. Conversely, if  $p_{0\dots r}(q)=0$ then the most left $(r+1)\times (r+1)$ sub-matrix of any matrix representing $q$ has zero determinant, therefore there is another representation of $q$ such that the first row has the following form $(0, \dots, 0,a_{r+1},\dots,a_n)$, and we conclude that $q\in D_0$. 

Similarly, setting $I_j=\{e_0,\dots,e_{j-1},e_{r+1},\dots,e_{n-j}\}$ for $j=0,\dots,r+1$, we have $D_j=V(p_{I_j})$. Note that $d(I_j,I_0)=j$ for any $j$. More generally, one can consider the prime divisor $D_I:=V(p_I)$ for any $I\in \Lambda$. We claim that the linear span of the orbit $\mathscr{G}\cdot D_j$ is given by
\stepcounter{thm}
\begin{equation}
\label{orbitofG}
\lin{(\mathscr{G}\cdot D_j)} = \left\langle \{
D_J; d(J,I_0)=j
\}\right\rangle, j=0,\dots, r+1
\end{equation}

Note that given $I,I'\in \Lambda$ with distance one and such that the non shared indexes, say $i\in I\backslash I', i'\in I'\backslash I$ are in $\{0,\dots, r\}$ we can find a $g\in\mathscr{G}$ such that $g(e_i)=e_{i'}$ and $g(e_j)=e_j$ for $j\neq i,i'$. The same holds if the non shared indexes are in $\{r+1,\dots, n\}$, and we get that 
$$\left\{
D_J; d(J,I_0)=j
\right\}\subset\mathscr{G}\cdot D_j, j=0,\dots, r+1
$$
Since the $D_J$ such that $d(J,I_0)=j, j=0,\dots, r+1$, give a generating set of $H^0(\mathbb{G}(r,n),\mathcal{O}_{\mathbb{G}(r,n)}(1))$ we get (\ref{orbitofG}). Now, let $S$ and $T_I$ be the canonical sections associated respectively to $E$ and $D_I$. By \cite[Theorem 4.5.4.6]{ADHL15} $S, T_I, I\in \Lambda$ are generators of 
$\mathcal{R}(\mathbb{G}(r,n)_1)$. Furthermore, by \cite[Lemma 7.2.1]{Ri17} for any $I\in \Lambda$ we have that 
$$\mult_{\left\langle e_0,\dots, e_r\right\rangle}D_I=1+\dim (\left\langle e_i, i\notin I\right\rangle\cap \left\langle e_0,\dots, e_r\right\rangle)= |(\{0,\dots,n\}\backslash I)
\cap \{0,\dots,r\}|=d(I,I_0)$$

Therefore, if $\deg(S) = (0,1)$ the degree of the other generators of $\mathcal{R}(\mathbb{G}(r,n)_1)$ in $\Pic(\mathbb{G}(r,n)_1)=\mathbb{Z}[H]\oplus\mathbb{Z}[E]$ is given by $\deg(T_I)=(1,-d(I,I_0))$.

The matrix representing this grading has size 
$2\times (N+1)$ and is of the following form
$$
A=\begin{pmatrix}
0 & 1 & 1 &\dots &  1 & 1 & \dots &1\\
1 & 0 & -1&\dots & -1 & -2& \dots &-(r+1)
\end{pmatrix}
$$
Our next aim is to find relations among the generators of $\mathcal{R}(\mathbb{G}(r,n)_1)$. Note that for each pair
$I=\{i_0<\dots < i_{r-1}\},
J=\{j_0<\dots <j_{r+1}\}\subset \{0,\dots,n\}$ with $|I|=r,|J|=r+2$, the polynomial
$$
G_{IJ}=\sum_{t=0}^{r+1}(-1)^t
T_{i_0\dots i_{r-1}j_t}
T_{j_1\dots \widehat{j_t}\dots j_{r+1}}
$$
is homogeneous of degree 
$(2,-|I\backslash I_0|-|J\backslash I_0|)$. Let $\mathfrak{J}\subset K[T_I,I\in \Lambda]$ be the ideal generated by the $G_{IJ}$. Since $\frac{K[T_I,I\in \Lambda]}{\mathfrak{J}}$ is the homogeneous coordinate ring of $\mathbb{G}(r,n)$, then 
$$\dim(K[T_I,I\in \Lambda]/ \mathfrak{J})=(r+1)(n-r)+1$$
and Remark \ref{dimCox} yields
$$\dim(K[S,T_I,I\in \Lambda]/ \mathfrak{J})\!=\!(r+1)(n-r)+2\!= \!
\dim(\mathbb{G}(r,n)_1)+
\rank(\Pic(\mathbb{G}(r,n)_1)))
\!=\!\dim(\mathcal{R}(\mathbb{G}(r,n)_1))$$
where we denote by $\mathfrak{J}$ the ideal generated by the polynomials $G_{IJ}$ in $K[S,T_I,I\in \Lambda]$. We conclude that there are no further relations in 
$\mathcal{R}(\mathbb{G}(r,n)_1)$ and hence $\mathcal{R}(\mathbb{G}(r,n)_1)=\frac{K[S,T_I,I\in \Lambda]}{\mathfrak{J}}$ as claimed. 
\end{proof}

Now, we are ready to compute the Mori chamber decomposition of $\Eff(\G(r,n)_1)$.

\begin{Proposition}
Let $\G(r,n)_1$ be the blow-up of the Grassmannian $\mathbb{G}(r,n)$ at a point. Then we have that $\Eff(\G(r,n)_1)  = \left\langle E, H-(r+1)E\right\rangle$, $\Nef(\G(r,n)_1)  =\left\langle H, H-E\right\rangle$ and
$$
\Mov(\G(r,n)_1)=\begin{cases}
\left\langle H,H-rE\right\rangle &\mbox{ if } n=2r+1;\\
\left\langle H,H-(r+1)E\right\rangle &\mbox{ if } n>2r+1.
\end{cases}
$$
Furthermore, $\MCD(\G(r,n)_1)$ and $\SBLD(\G(r,n)_1)$ coincide and their walls are given by the divisors $E,H,H-E,\dots,H-(r+1)E$ as represented in the following picture
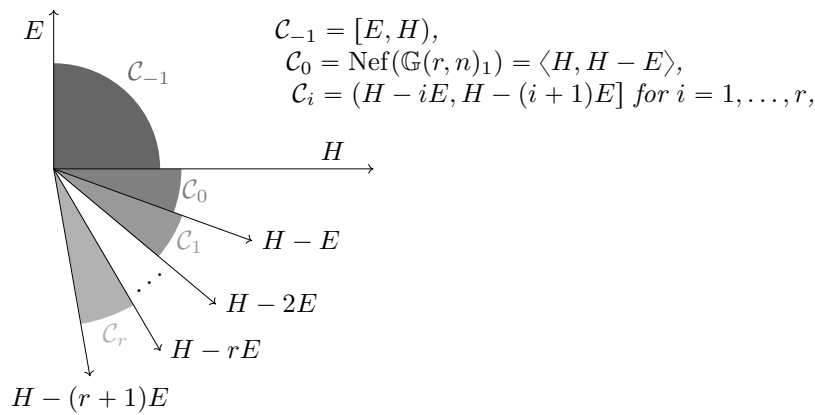
\vspace{-0.2cm}
$$
\begin{tikzpicture}[scale=1.4]
\fill[gray0] (0,0) -- (1.0,0)
arc [start angle=0, end angle=90,
radius=1.0];
\fill[gray] (0,0) -- (1.2,0)
arc [start angle=0, end angle=-20, radius=1.2];
\fill[gray1] (0,0) -- (0.93*1.3,-0.34*1.3)
arc [start angle=-20, end angle=-40, radius=1.3];
\fill[gray2] (0,0) -- (0.5*1.5,-0.86*1.5)
arc [start angle=-60, end angle=-80, radius=1.5];
\draw[->][line width=0mm] (0,0) -- (0,1.5)   node[left,very near end]{$E$};
\draw[->][line width=0mm] (0,0) -- (3,0)   node[above,very near end]{$H$};
\draw (0.9,0.9) node[thick,gray0]{$\mathcal{C}_{-1}$};
\draw[->][line width=0mm] (0,0) -- (0.93*2,-0.34*2) 
node[right,thick]{$H-E$}
node[above right, pos=0.6, thick,gray]{$\mathcal C_0$}; 
\draw[->][line width=0mm] (0,0) -- (0.76*2,-0.64*2) 
node[right,thick]{$H-2E$}
node[above right, pos=0.7, thick,gray1]{$\mathcal C_1$}; \draw (0.9,-1) node[thick]{$\rddots$};
\draw[->][line width=0mm] (0,0) -- (0.5*2,-0.86*2)
node[right,thick]{$H-rE$}
node[below left, pos=0.8, thick,gray2]{$\mathcal C_r$};  
\draw[->][line width=0mm] (0,0) -- (0.17*2,-0.98*2)
node[below,thick]{$H-(r+1)E$};
\draw (2.85,1.3) node[thick]{$\mathcal{C}_{-1} =  [E,H)$,};
\draw (4.05,1) node[thick]{$\mathcal{C}_0 = \Nef(\G(r,n)_1) = \left\langle H,H-E\right\rangle$,};
\draw (4.7,0.7) node[thick]{$\mathcal{C}_i = (H-iE,H-(i+1)E]$ for $i = 1,\dots,r$,};
\end{tikzpicture}
$$
where with the notation $\mathcal{C}_i = (H-iE,H-(i+1)E]$ we mean that the ray spanned by $H-(i+1)E$ belongs to $\mathcal{C}_i$ but the ray spanned by $H-iE$ does not, and similarly with the notation $\mathcal{C}_{-1} = [E,H)$ we mean that the ray spanned by $E$ belongs to $\mathcal{C}_{-1}$ but the ray spanned by $H$ does not.
\end{Proposition}
\begin{proof}
The claims on the effective, nef and movable cones follow from \cite[Theorem 1.3]{MR18}. Furthermore, by \cite[Theorem 1.3]{MR18} the decomposition displayed in the statement is the stable base locus decomposition of $\Eff(\G(r,n)_1)$. Now, by Proposition \ref{CoxG1} all the generators of $\mathcal{R}(\G(r,n)_1)$ appear in the walls of the stable base locus decomposition of $\Eff(\G(r,n)_1)$, and then Proposition \ref{crit1} yields that the Mori chamber and the stable base locus decomposition of $\Eff(\G(r,n)_1)$ coincide. 
\end{proof}

\bibliographystyle{amsalpha}
\bibliography{Biblio}

\end{document}